\ifdefined\BUILDICML
\def\MODE{2}
\else
\def\MODE{1}
\fi

\documentclass{article}

\if\MODE2
\usepackage{times}
\fi
\usepackage{graphicx} 
\usepackage{subfigure}

\usepackage{algorithm}
\usepackage{algorithmic}

\if\MODE1
\newcommand{\COND}[2]{#1}
\else
\newcommand{\COND}[2]{#2}
\fi

\COND{}{\usepackage{natbib}}

\usepackage[colorlinks=true,citecolor=blue,linkcolor=blue]{hyperref}

\if\MODE2
    \usepackage{icml2016}
    
\else
    \usepackage{fullpage} 
\fi

\usepackage{amsmath,amsfonts,amssymb}
\usepackage{amsthm}
\usepackage{mathtools}
\usepackage{enumerate}
\usepackage{bbm}
\usepackage{bm}
\usepackage{appendix}

\newcommand*\samethanks[1][\value{footnote}]{\footnotemark[#1]}
\newcommand{\SUPsubsection}{\COND{\subsubsection}{\subsection}}

    \numberwithin{equation}{section}

\newtheorem{theorem}{Theorem}[section]
\newtheorem{lemma}[theorem]{Lemma}

\newtheorem{definition}[theorem]{Definition}

\newcommand{\vct}[1]{\bm{#1}}
\newcommand{\mtx}[1]{\bm{#1}}

\DeclareMathOperator*{\rank}{rank}
\newcommand{\Tr}{\mathrm{Tr}}

\newcommand{\beq}{\begin{equation}}
\newcommand{\eeq}{\end{equation}}
\newcommand{\beqs}{\begin{equation*}}
\newcommand{\eeqs}{\end{equation*}}
\newcommand{\dist}{\mathrm{dist}}

\newcommand{\A}{\mathcal{A}}
\newcommand{\R}{\mathbb{R}}

\newcommand{\ip}[2]{\ensuremath{\langle #1, #2 \rangle}}
\newcommand{\Normal}{\mathcal{N}}
\newcommand{\norm}[1]{\lVert #1 \rVert}
\newcommand{\opnorm}[1]{\left\|#1\right\|}
\newcommand{\fronorm}[1]{\left\|#1\right\|_{F}}
\newcommand{\twonorm}[1]{\left\|#1\right\|_{\ell_2}}
\newcommand{\abs}[1]{\ensuremath{| #1 |}}
\newcommand{\BigO}{\mathcal{O}}
\newcommand{\BigOmega}{\Omega}

\newcommand{\RC}{\mathsf{RC}}

\newcommand{\T}{\mathsf{T}}

\newcommand{\UMinusXR}{\mtx{U}-\mtx{XR}}
\newcommand{\WMinusZR}{\mtx{W}-\mtx{ZR}}
\newcommand{\UUTMinusXXT}{\mtx{UU}^\T-\mtx{XX}^\T}
\newcommand{\WWTMinusZZT}{\mtx{WW}^\T-\mtx{ZZ}^\T}
\newcommand{\UVTMinusM}{\mtx{UV}^\T-\mtx{M}}
\newcommand{\UUT}{\mtx{UU}^\T}
\newcommand{\UVT}{\mtx{UV}^\T}

\newcommand{\XXT}{\mtx{XX}^\T}
\newcommand{\XYT}{\mtx{XY}^\T}

\newcommand{\tinit}{T_{0}}

\newcommand{\Sym}{\mathrm{Sym}}
\newcommand{\blockvec}[2]{\begin{bmatrix} #1 \\ #2 \end{bmatrix}}
\newcommand{\blockmat}[4]{\begin{bmatrix} #1 & #2 \\ #3 & #4 \end{bmatrix}}
\newcommand{\blockmatoff}[2]{\blockmat{\mtx{0}}{#1}{#2}{\mtx{0}}}
\newcommand{\blockmatdiag}[2]{\blockmat{#1}{\mtx{0}}{\mtx{0}}{#2}}
\newcommand{\Projdiag}{\mathcal{P}_{\mathrm{diag}}}
\newcommand{\Projoff}{\mathcal{P}_{\mathrm{off}}}

\newcommand{\B}{\mathcal{B}}

\if\MODE2
    \icmltitlerunning{Low-rank Solutions of Linear Matrix Equations via Procrustes Flow}
\else
    \title{Low-rank Solutions of Linear Matrix Equations\\ via Procrustes Flow}
    \author{
      Stephen Tu~\thanks{Department of Electrical Engineering and Computer Science, UC Berkeley, Berkeley CA.} \quad
      Ross Boczar~\samethanks[1] \quad
      Max Simchowitz~\samethanks[1] \\
      Mahdi Soltanolkotabi~\thanks{Ming Hsieh Department of Electrical Engineering, University of Southern California, Los Angeles, CA.} \quad
      Benjamin Recht~\samethanks[1]~\thanks{Department of Statistics, UC Berkeley, Berkeley CA.}
    }
    \date{\today}
\fi

\begin{document}

\if\MODE2
    \twocolumn[
    \icmltitle{Low-rank Solutions of Linear Matrix Equations\\ via Procrustes Flow}
    \icmlauthor{Stephen Tu,~Ross Boczar,~Max Simchowitz}{\{stephent,boczar,msimchow\}@berkeley.edu}
    \icmladdress{Department of Electrical Engineering and Computer Science, UC Berkeley, Berkeley CA}
    \icmlauthor{Mahdi Soltanokotabi}{soltanol@usc.edu}
    \icmladdress{Ming Hsieh Department of Electrical Engineering, University of Southern California, Los Angeles, CA}
    \icmlauthor{Benjamin Recht}{brecht@berkeley.edu}
    \icmladdress{Department of Statistics, UC Berkeley, Berkeley CA}

    \icmlkeywords{KEYWORDS TODO}
    \vskip 0.3in
    ]
\else
    \maketitle
    \thispagestyle{empty}
\fi

\begin{abstract}
In this paper we study the problem of recovering a low-rank 
matrix from linear measurements. Our algorithm, which we call
\emph{Procrustes Flow}, starts from an initial estimate obtained by a
thresholding scheme followed by gradient descent on a non-convex objective. We
show that as long as the measurements obey a standard restricted isometry
property, our algorithm converges to the unknown matrix at a geometric rate.
In the case of Gaussian measurements, such convergence occurs for a $n_1 \times n_2$
matrix of rank $r$ when the number of measurements exceeds a constant times $(n_1+n_2)r$.
\end{abstract}

\section{Introduction}
Low rank models are ubiquitous in machine learning, and over a decade of
research has been dedicated to determining when such models can be efficiently
recovered from partial information~\cite{fazel02,Rennie05,candes2009exact}.
See~\cite{davenport16} for an extended survey on this topic.
The simplest such recovery problem concerns how can we can find a low-rank
matrix obeying a set of linear equations? What is the computational complexity
of such an algorithm? More specifically, we are interested in solving problems
of the form
\begin{align}
\label{mainopt}
\underset{\quad\mtx{M} \in \R^{n_1 \times n_2}}{\min}\text{ }\text{rank}(\mtx{M})\quad\text{s.t.}\quad\A(\mtx{M}) = \vct{b}\:,
  \end{align}
where $\A :\R^{n_1\times n_2}\longrightarrow \R^m$ is a known affine transformation
that maps matrices to vectors. More specifically, the $k$-th entry of
$\A(\mtx{X})$ is $\ip{\mtx{A}_k}{\mtx{X}} :=
\Tr(\mtx{A}_k^\T \mtx{X})$, where each $\mtx{A}_k\in\R^{n_1\times n_2}$.

Since the early seventies, a popular heuristic for solving such problems
has been to replace $\mtx{M}$ with a low-rank factorization
$\mtx{M}=\mtx{U}\mtx{V}^\T$ and solve matrix bilinear equations of the form
\begin{align}
  \underset{\mtx{U} \in \R^{n_1 \times r}, \mtx{V} \in \R^{n_2 \times r} }{\text{find}}\quad\text{s.t.}\quad\A(\mtx{U}\mtx{V}^\T) = \vct{b}, \label{eq:psdequiv}
\end{align}
via a local search heuristic~\cite{Ruhe74}.  Many researchers have demonstrated
that such heuristics work well in practice for a variety of
problems~\cite{Rennie05,SimonFunk,lee10b,recht13}.  However, these procedures
lack strong guarantees associated with convex programming heuristics for
solving~\eqref{mainopt}.

In this paper we show that a local search heuristic solves
\eqref{eq:psdequiv} under standard restricted isometry assumptions on the
linear map $\A$.  For standard ensembles of equality constraints, we
demonstrate that $\mtx{M}$ can be estimated by such
heuristics as long as we have $\BigOmega((n_1+n_2)r)$ equations.\footnote{Here and throughout we use $f(x) = \BigOmega(g(x))$ if there is a
positive constant $C$ such that $f(x) \geq C g(x)$ for all $x$ sufficiently
large.}
This is merely a constant factor more than the number of parameters needed to
specify a $n_1\times n_2$ rank $r$ matrix.  Specialized to a random Gaussian model and positive semidefinite matrices,
our work improves upon recent independent work by Zheng and Lafferty
\cite{zheng15}.

\section{Algorithms}

In this paper we study a local search heuristic for solving matrix bilinear
equations of the form \eqref{eq:psdequiv} which consists of two components: (1)
a careful initialization obtained by a projected gradient scheme on $n_1\times
n_2$ matrices, and (2) a series of successive refinements of this initial
solution via a gradient descent scheme. This algorithm is a natural extension
of the Wirtinger Flow algorithm developed in \cite{candes14} for solving vector
quadratic equations.
Following \cite{candes14}, we shall refer to the combination of these two steps
as the Procrustes Flow (PF) algorithm.  We shall describe two variants of our
algorithm based on whether the sought after solution $\mtx{M}$ is
positive semidefinite or not. The former is detailed in
Algorithm~\ref{alg:pf}, and the latter in Algorithm~\ref{alg:rpf}.

The initialization phase of both variants is rather similar and is described in
Section \ref{initsec}. The successive refinement phase is explained in Section
\ref{refinePSD} for positive semidefinite (PSD) matrices and in Section \ref{refinenonPSD} for
arbitrary matrices. Throughout this paper when describing the PSD case, we
assume the size of the matrix is $\mtx{M}$ is $n\times n$, i.e.~$n_1=n_2=n$.

\subsection{Initialization via low-rank projected gradients}\label{initsec}
In the initial phase of our algorithm we start from $\widetilde{\mtx{M}}_0=\mtx{0}_{n_1
\times n_2}$ and apply successive updates of the form
\begin{align}
\label{update}
\widetilde{\mtx{M}}_{\tau+1}=\mathcal{P}_r\left(\widetilde{\mtx{M}}_\tau-\alpha_{\tau+1}\sum_{k=1}^m\left(\langle \mtx{A}_k,\widetilde{\mtx{M}}_\tau\rangle-b_k\right)\mtx{A}_k\right),
\end{align}
on rank $r$ matrices of size $n_1\times n_2$.  Here, $\mathcal{P}_r$ denotes
projection onto either rank-$r$ matrices or rank-$r$ PSD matrices, both of
which can be computed efficiently via Lanczos methods. We run \eqref{update}
for $\tinit$ iterations and use the resulting matrix
$\mtx{M}_{\tinit}$ for initialization purposes.  In the PSD
case, we set our initialization to an $n\times r$ matrix $\mtx{U}_0$ obeying
$\widetilde{\mtx{M}}_{\tinit}=\mtx{U}_0\mtx{U}_0^\T$. In the more general case
of rectangular matrices we need to use two factors. Let
$\widetilde{\mtx{M}}_{\tinit} = \mtx{C}_{\tinit} \mtx{\Sigma}_{\tinit}
\mtx{D}_{\tinit}^\T$ be the Singular Value Decomposition (SVD) of
$\widetilde{\mtx{M}}_{\tinit}$. We initialize our algorithm in the rectangular
case by setting $\mtx{U}_0 = \mtx{C}_{\tinit} \mtx{\Sigma}_{\tinit}^{1/2}$ and
$\mtx{V}_0 = \mtx{D}_{\tinit} \mtx{\Sigma}_{\tinit}^{1/2}$.

Updates of the form \eqref{update} have a long history in compressed
sensing/matrix sensing literature (see e.g. \cite{OMP2,
garg2009gradient,CoSamp, needell2009uniform,blumensath2009, meka09,
cai2010singular}). Furthermore, using the first step of the update
\eqref{update} for the purposes of initialization has also been proposed in
previous work (see e.g. \cite{achlioptas2007fast, keshavan2010matrix, jain12}).

\subsection{Successive refinement via gradient descent -- positive semidefinite case}
\label{refinePSD}
We first focus on the PSD case.
As mentioned earlier, we are interested in finding a matrix
$\mtx{U}\in\R^{n\times r}$ obeying matrix quadratic equations of the form
$\A(\UUT)= \vct{b}$.
\COND{
We wish to refine our initial estimate by solving the
non-convex optimization problem
\begin{align}
\label{eq:theproblem}
\min_{\mtx{U} \in \R^{n \times r}} f(\mtx{U}) := \frac{1}{4} \twonorm{ \A(\UUT) - \vct{b} }^2 = \frac{1}{4} \sum_{k=1}^{m} ( \ip{\mtx{A}_k}{\UUT} - b_k)^2,
\end{align}
}{
We wish to refine our initial estimate by minimizing the
non-convex function
\begin{align}
\label{eq:theproblem}
 f(\mtx{U}) := \frac{1}{4} \twonorm{ \A(\UUT) - \vct{b} }^2,
\end{align}
over $\mtx{U} \in \R^{n \times r}$,
}
which minimizes the misfit in our quadratic equations via the square loss. To
solve \eqref{eq:theproblem}, starting from our initial estimate
$\mtx{U}_0\in\R^{n\times r}$ we apply the successive updates
\begin{align}
\label{graddescentupdate}
\mtx{U}_{\tau+1} &:= \mtx{U}_\tau-\frac{\mu_{\tau+1}}{\opnorm{\mtx{U}_0}^2}\nabla f(\mtx{U}_\tau) \COND{}{\notag \\&}= \mtx{U}_\tau-\frac{\mu_{\tau+1}}{\opnorm{\mtx{U}_0}^2}\left(\sum_{k=1}^{m} (\ip{\mtx{A}_k}{\mtx{U}_\tau \mtx{U}_\tau^\T} - b_k) \mtx{A}_k \mtx{U}_\tau\right).
\end{align}
Here and throughout, for a matrix $\mtx{X}$, $\sigma_\ell(\mtx{X})$ denotes the
$\ell$-th largest singular value of $\mtx{X}$, and $\opnorm{\mtx{X}} =
\sigma_1(\mtx{X})$ is the operator norm.  We note that the update
\eqref{graddescentupdate} is essentially gradient descent with a carefully
chosen step size.

\begin{algorithm}
\caption{Procrustes Flow (PF)}
\label{alg:pf}
\begin{algorithmic}
    \REQUIRE \mbox{$\{\mtx{A}_k\}_{k=1}^{m}, \{b_k\}_{k=1}^{m}, \{\alpha_\tau\}_{\tau=1}^{\infty}, \{\mu_\tau\}_{\tau=1}^\infty, \tinit \in\mathbb{N}$}.
\STATE \texttt{// Initialization phase.}
\STATE $\widetilde{\mtx{M}}_0 := \mtx{0}_{n \times n}$.
\FOR{$\tau=0, 1, ..., \tinit - 1$}
\STATE \texttt{// Projection onto rank $r$ PSD matrices.}
\STATE \mbox{$\widetilde{\mtx{M}}_{\tau+1} \gets \mathcal{P}_{r}(\widetilde{\mtx{M}}_\tau - \alpha_{\tau+1} \sum_{k=1}^{m} (\ip{\mtx{A}_k}{\widetilde{\mtx{M}}_\tau} - b_k) \mtx{A}_k)$}.
\ENDFOR
\STATE \texttt{// SVD of $\widetilde{\mtx{M}}_{\tinit}$, with $\mtx{Q} \in \R^{n \times r}, \mtx{\Sigma} \in \R^{r \times r}$.}
\STATE $\mtx{Q} \mtx{\Sigma} \mtx{Q}^\T := \widetilde{\mtx{M}}_{\tinit}$.
\STATE $\mtx{U}_0 := \mtx{Q} \mtx{\Sigma}^{1/2}$.
\STATE \texttt{// Gradient descent phase.}
\REPEAT
\STATE
    \COND{\mbox{$\mtx{U}_{\tau+1} \gets \mtx{U}_\tau - \frac{\mu_{\tau+1}}{\opnorm{\mtx{U}_0}^2}\left(\sum_{k=1}^{m} (\ip{\mtx{A}_k}{\mtx{U}_\tau \mtx{U}_\tau^\T} - b_k) \mtx{A}_k \mtx{U}_\tau\right)$}.}
    {\mbox{$\mtx{U}_{\tau+1} \gets \mtx{U}_\tau-\frac{\mu_{\tau+1}}{\opnorm{\mtx{U}_0}^2}\nabla f(\mtx{U}_\tau)$}.}
\UNTIL{convergence}
\end{algorithmic}
\end{algorithm}

\subsection{Successive refinement via gradient descent -- general case}
\label{refinenonPSD}
We now consider the general case.  Here, we are interested in finding matrices
$\mtx{U}\in\R^{n_1\times r}$ and $\mtx{V}\in\R^{n_2\times r}$ obeying matrix quadratic equations of the form
$\vct{b}=\A(\mtx{U}\mtx{V}^\T)$.
\COND{
In this case, we refine our initial estimate by solving the
non-convex optimization problem
\begin{align}
\label{eq:therectproblem}
\min_{\mtx{U} \in \R^{n_1 \times r},\: \mtx{V} \in \R^{n_2 \times r}}  g(\mtx{U},\mtx{V})
:= \frac{1}{2} \twonorm{ \A(\mtx{U}\mtx{V}^\T) - \vct{b} }^2  + \frac{1}{16}\fronorm{\mtx{U}^\T\mtx{U} - \mtx{V}^\T\mtx{V}}^2\:.
\end{align}
}{
In this case, we refine our initial estimate by minimizing the
non-convex function
\begin{align}
\label{eq:therectproblem}
g(\mtx{U},\mtx{V}) := \frac{1}{2} \twonorm{ \A(\mtx{U}\mtx{V}^\T) - \vct{b} }^2  + \frac{1}{16}\fronorm{\mtx{U}^\T\mtx{U} - \mtx{V}^\T\mtx{V}}^2\:.
\end{align}
over $\mtx{U} \in \R^{n_1 \times r}$ and $\mtx{V} \in \R^{n_2 \times r}$.
}
Note that this is similar to \eqref{eq:theproblem} but adds a regularizer to
measure mismatch between $\mtx{U}$ and $\mtx{V}$. Given a factorization
$\mtx{M}=\mtx{U}\mtx{V}^\T$, for any invertible $r\times r$ matrix
$\mtx{P}$, $\mtx{U}\mtx{P}$ and $\mtx{V}\mtx{P}^{-\T}$ is also a valid
factorization. The purpose of the second term in \eqref{eq:therectproblem} is
to account for this redundancy and put the two factors on ``equal footing''.
\COND{
To solve \eqref{eq:therectproblem}, starting from our initial estimates
$\mtx{U}_0$ and $\mtx{V}_0$ we apply the successive updates
\begin{align}
\label{rectgraddescentupdate_a}
\mtx{U}_{\tau+1}:= &\; \mtx{U}_\tau-\frac{\mu_{\tau+1}}{\opnorm{\mtx{U}_0}^2}\nabla_{\mtx{U}} g(\mtx{U}_\tau,\mtx{V}_\tau) \notag\\
= &\;  \mtx{U}_\tau-\frac{\mu_{\tau+1}}{\opnorm{\mtx{U}_0}^2}\left(\sum_{k=1}^{m} (\ip{\mtx{A}_k}{\mtx{U}_\tau \mtx{V}_\tau^\T} - b_k) \mtx{A}_k \mtx{V}_\tau + \frac{1}{4}\mtx{U}_\tau(\mtx{U}_\tau^\T\mtx{U}_\tau-\mtx{V}_\tau^\T\mtx{V}_\tau)\right)
\end{align}
\begin{align}
\label{rectgraddescentupdate_b}
\mtx{V}_{\tau+1}:= &\; \mtx{V}_\tau-\frac{\mu_{\tau+1}}{\opnorm{\mtx{V}_0}^2}\nabla_{\mtx{V}} g(\mtx{U}_\tau,\mtx{V}_\tau) \notag\\
= &\;  \mtx{V}_\tau-\frac{\mu_{\tau+1}}{\opnorm{\mtx{V}_0}^2}\left(\sum_{k=1}^{m} (\ip{\mtx{A}_k}{\mtx{U}_\tau \mtx{V}_\tau^\T} - b_k) \mtx{A}_k^\T \mtx{U}_\tau + \frac{1}{4}\mtx{V}_\tau(\mtx{V}_\tau^\T\mtx{V}_\tau-\mtx{U}_\tau^\T\mtx{U}_\tau)\right).
\end{align}
}{
To solve \eqref{eq:therectproblem}, starting from our initial estimates
$\mtx{U}_0$ and $\mtx{V}_0$ we apply the successive updates
\begin{align}
    \mtx{U}_{\tau+1}:= &\; \mtx{U}_\tau-\frac{\mu_{\tau+1}}{\opnorm{\mtx{U}_0}^2}\nabla_{\mtx{U}} g(\mtx{U}_\tau,\mtx{V}_\tau) \label{rectgraddescentupdate_a} \\
    \mtx{V}_{\tau+1}:= &\; \mtx{V}_\tau-\frac{\mu_{\tau+1}}{\opnorm{\mtx{V}_0}^2}\nabla_{\mtx{V}} g(\mtx{U}_\tau,\mtx{V}_\tau) \label{rectgraddescentupdate_b}
\end{align}
where $\nabla_{\mtx{U}} g(\mtx{U}_\tau,\mtx{V}_\tau)$ is equal to
\begin{align*}
    \sum_{k=1}^{m} (\ip{\mtx{A}_k}{\mtx{U}_\tau \mtx{V}_\tau^\T} - b_k) \mtx{A}_k \mtx{V}_\tau + \frac{1}{4}\mtx{U}_\tau(\mtx{U}_\tau^\T\mtx{U}_\tau-\mtx{V}_\tau^\T\mtx{V}_\tau)
\end{align*}
and $\nabla_{\mtx{V}} g(\mtx{U}_\tau,\mtx{V}_\tau)$ is equal to
\begin{align*}
    \sum_{k=1}^{m} (\ip{\mtx{A}_k}{\mtx{U}_\tau \mtx{V}_\tau^\T} - b_k) \mtx{A}_k^\T \mtx{U}_\tau + \frac{1}{4}\mtx{V}_\tau(\mtx{V}_\tau^\T\mtx{V}_\tau-\mtx{U}_\tau^\T\mtx{U}_\tau) \:.
\end{align*}
}
Again, \eqref{rectgraddescentupdate_a} and \eqref{rectgraddescentupdate_b} are
essentially gradient descent with a carefully chosen step size.
\begin{algorithm}
\caption{Rectangular Procrustes Flow (RPF)}
\label{alg:rpf}
\begin{algorithmic}
    \REQUIRE \mbox{$\{\mtx{A}_k\}_{k=1}^{m}, \{b_k\}_{k=1}^{m}, \{\alpha_\tau\}_{\tau=1}^{\infty}, \{\mu_\tau\}_{\tau=1}^\infty, \tinit \in\mathbb{N}$}.
\STATE \texttt{// Initialization phase.}
\STATE $\widetilde{\mtx{M}}_0 := \mtx{0}_{n_1 \times n_2}$.
\FOR{$\tau = 0, 1, ..., \tinit - 1$}
\STATE \texttt{// Projection onto rank $r$ matrices.}
\STATE \mbox{$\widetilde{\mtx{M}}_{\tau+1} \gets \mathcal{P}_{r}(\widetilde{\mtx{M}}_\tau - \alpha_{\tau+1} \sum_{k=1}^{m} (\ip{\mtx{A}_k}{\widetilde{\mtx{M}}_\tau} - b_k) \mtx{A}_k)$}.
\ENDFOR
\COND{
\STATE \texttt{// SVD of $\widetilde{\mtx{M}}_{\tinit}$, with $\mtx{C} \in \R^{n_1 \times r}, \mtx{\Sigma} \in \R^{r \times r}, \mtx{D} \in \R^{n_2 \times r}$ .}
}{
\STATE \texttt{// SVD of $\widetilde{\mtx{M}}_{\tinit}$, with}
\STATE \texttt{// $\mtx{C} \in \R^{n_1 \times r}, \mtx{\Sigma} \in \R^{r \times r}, \mtx{D} \in \R^{n_2 \times r}$ .}
}
\STATE $\mtx{C} \mtx{\Sigma} \mtx{D}^\T := \widetilde{\mtx{M}}_{\tinit}$.
\STATE $\mtx{U}_0 := \mtx{C} \mtx{\Sigma}^{1/2}$.
\STATE $\mtx{V}_0 := \mtx{D} \mtx{\Sigma}^{1/2}$.
\STATE \texttt{// Gradient descent phase.}
\REPEAT
\COND{
\STATE $\mtx{U}_{\tau+1} \gets \mtx{U}_\tau-\mu_{\tau+1}\frac{1}{\opnorm{\mtx{U}_0}^2}\left(\sum_{k=1}^{m} (\ip{\mtx{A}_k}{\mtx{U}_\tau \mtx{V}_\tau^\T} - b_k) \mtx{A}_k \mtx{V}_\tau + \frac{1}{4}\mtx{U}_\tau(\mtx{U}_\tau^\T\mtx{U}_\tau-\mtx{V}_\tau^\T\mtx{V}_\tau)\right)$.
\STATE $\mtx{V}_{\tau+1} \gets \mtx{V}_\tau-\mu_{\tau+1}\frac{1}{\opnorm{\mtx{V}_0}^2}\left(\sum_{k=1}^{m} (\ip{\mtx{A}_k}{\mtx{U}_\tau \mtx{V}_\tau^\T} - b_k) \mtx{A}_k^\T \mtx{U}_\tau + \frac{1}{4}\mtx{V}_\tau(\mtx{V}_\tau^\T\mtx{V}_\tau-\mtx{U}_\tau^\T\mtx{U}_\tau)\right)$.
}{
\STATE $\mtx{U}_{\tau+1} \gets \mtx{U}_\tau-\frac{\mu_{\tau+1}}{\opnorm{\mtx{U}_0}^2}\nabla_{\mtx{U}} g(\mtx{U}_\tau,\mtx{V}_\tau)$.
\STATE $\mtx{V}_{\tau+1} \gets \mtx{V}_\tau-\frac{\mu_{\tau+1}}{\opnorm{\mtx{V}_0}^2}\nabla_{\mtx{V}} g(\mtx{U}_\tau,\mtx{V}_\tau)$.
}
\UNTIL{convergence}
\end{algorithmic}
\end{algorithm}

\section{Main Results}
For our theoretical results we shall focus on affine maps $\mathcal{A}$ which
obey the matrix Restricted Isometry Property (RIP).
\COND{\pagebreak}{}
\begin{definition}[Restricted Isometry Property (RIP) \cite{candes2005decoding, recht10}]
The map $\A$ satisfies $r$-RIP with constant $\delta_r$, if
\begin{align*}
  (1-\delta_r) \fronorm{\mtx{X}}^2 \leq \twonorm{ \A(\mtx{X}) }^2 \leq (1+\delta_r) \fronorm{\mtx{X}}^2,
\end{align*}
holds for all matrices $\mtx{X}\in\R^{n_1\times n_2}$ of rank at most $r$.
\end{definition}
As mentioned earlier it is not possible to recover the factors $\mtx{U}$ and $\mtx{V}$ in \eqref{eq:psdequiv} exactly. For example, in the PSD case it is only possible to recover $\mtx{U}$ up to a certain
rotational factor as if $\mtx{U}$ obeys \eqref{mateq}, then so does any matrix
$\mtx{U}\mtx{R}$ with $\mtx{R}\in\R^{r\times r}$ an orthonormal matrix satisfying
$\mtx{R}^\T \mtx{R}=\mtx{I}_r$.
This naturally leads to defining the distance
between two matrices $\mtx{U},\mtx{X}\in\R^{n\times r}$ as
\begin{align}
\label{procrustesprob}
\text{dist}(\mtx{U},\mtx{X}):=\min_{\mtx{R} \in \R^{r\times r}:\text{ }\mtx{R}^\T\mtx{R}=\mtx{I}_r} \fronorm{\mtx{U} - \mtx{XR}}.
\end{align}
We note that this distance is the solution to the classic \emph{orthogonal
Procrustes problem} (hence the name of the algorithm). It is known that the
optimal rotation matrix $\mtx{R}$ minimizing $\fronorm{\mtx{U} - \mtx{XR}}$ is
equal to $\mtx{R}=\mtx{A}\mtx{B}^\T$, where $\mtx{A}\mtx{\Sigma}\mtx{B}^\T$ is
the singular value decomposition (SVD) of $\mtx{X}^\T\mtx{U}$. We now have all
of the elements in place to state our main results.

\subsection{Quadratic measurements}
When the low-rank matrix $\mtx{M}\in\R^{n\times n}$ is PSD we are interested in finding a matrix $\mtx{U}\in\R^{n\times r}$ obeying quadratic equations of the form
\begin{align}
\label{mateqPSD}
\A( \mtx{U}\mtx{U}^T  )=\vct{b},
\end{align}
where we assume $\vct{b}=\mathcal{A}(\mtx{M})$ for a planted rank-$r$ solution $\mtx{M}=\mtx{X}\mtx{X}^\T\in\R^{n\times n}$ with $\mtx{X}\in\R^{n\times r}$. We wish to recover $\mtx{X}$. This is of course only possible up to a certain
rotational factor as if $\mtx{U}$ obeys \eqref{mateq}, then so does any matrix
$\mtx{U}\mtx{R}$ with $\mtx{R}\in\R^{r\times r}$ an orthonormal matrix satisfying
$\mtx{R}^\T \mtx{R}=\mtx{I}_r$. Our first theorem shows that Procrustes Flow indeed recovers $\mtx{X}$ up to this ambiguity factor.
\begin{theorem}
\label{mainthm} Let $\mtx{M}\in\R^{n\times n}$ be an arbitrary rank-$r$ symmetric positive semidefinite matrix with
singular values $\sigma_1(\mtx{M}) \geq \sigma_2(\mtx{M}) \geq ... \geq
\sigma_r(\mtx{M}) > 0$ and condition number $\kappa =
\sigma_1(\mtx{M})/\sigma_r(\mtx{M})$. Assume $\mtx{M} = \XXT$ for some $\mtx{X} \in \R^{n \times r}$ and let
$\vct{b}=\A(\mtx{M})\in\R^m$ be $m$ linear measurements.
Furthermore, assume the mapping $\mathcal{A}$ obeys rank-$6r$
RIP with RIP constant $\delta_{6r} \le 1/10$. Also let $\alpha_\tau=1/m$ for all $\tau=1,2,\ldots$.
Then, using $\tinit \ge
\log(\sqrt{r}\kappa)+2$ iterations of the initialization phase of Procrustes
Flow as stated in Algorithm~\ref{alg:pf} yields a solution $\mtx{U}_0$ obeying
\begin{align}
\label{initthm}
\emph{dist}\left(\mtx{U}_0,\mtx{X}\right)\le \frac{1}{4}\sigma_r(\mtx{X}).
\end{align}
Furthermore, take a constant step size $\mu_\tau=\mu$ for all $\tau=1,2,\ldots$,
with $\mu \le 36/425$.  Then, starting from any initial solution obeying
\eqref{initthm}, the $\tau$-th iterate of Algorithm~\ref{alg:pf} satisfies
\begin{align}
\label{convthm}
\emph{dist}\left(\mtx{U}_\tau,\mtx{X}\right)\le \frac{1}{4}\left(1-  \frac{8}{25} \frac{\mu}{\kappa}\right)^{\frac{\tau}{2}}\sigma_r(\mtx{X}).
\end{align}
\end{theorem}

\subsection{Bilinear measurements}
In the more general case when the low-rank matrix $\mtx{M}\in\R^{n_1\times n_2}$ is rectangular we are interested in finding matrices $\mtx{U}\in\R^{n_1\times r}$, $\mtx{V} \in \R^{n_2\times r}$ obeying bilinear equations of the form
\begin{align}
\label{mateq}
\A( \UVT  )=\vct{b},
\end{align}
where we assume $\vct{b}=\A(\mtx{M})$ for a planted rank-$r$ solution
$\mtx{M}\in\R^{n_1\times n_2}$ with $\mtx{M}=\mtx{X}\mtx{Y}^\T$ where
$\mtx{X}\in\R^{n_1\times r}$ and $\mtx{Y}\in\R^{n_2\times r}$.  Again we wish
to recover the factors $\mtx{X}$ and $\mtx{Y}$.  The next theorem shows that we
can also provide a guarantee similar to that of Theorem \ref{mainthm} for this
more general rectangular case.
\begin{theorem}
\label{mainthm_general}
Let $\mtx{M}\in\R^{n_1\times n_2}$ be an arbitrary rank-$r$ matrix with
singular values $\sigma_1(\mtx{M}) \geq \sigma_2(\mtx{M}) \geq ... \geq
\sigma_r(\mtx{M}) > 0$ and condition number $\kappa =
\sigma_1(\mtx{M})/\sigma_r(\mtx{M})$. Let $\mtx{M} = \mtx{A}\mtx{\Sigma}\mtx{B}^\T$
be the SVD of $\mtx{M}$ and define $\mtx{X} = \mtx{A} \mtx{\Sigma}^{1/2} \in \R^{n_1 \times r}$ and
$\mtx{Y} = \mtx{B} \mtx{\Sigma}^{1/2} \in \R^{n_2 \times r}$.
Also, let
$\vct{b}=\A(\mtx{M})\in\R^m$ be $m$  linear measurements where the mapping $\mathcal{A}$ obeys rank-$6r$
RIP with RIP constant $\delta_{6r} \le 1/25$. Also let $\alpha_\tau=1/m$ for all $\tau=1,2,\ldots$.
Then, using $\tinit \ge
3\log(\sqrt{r}\kappa)+5$ iterations of the initialization phase of Procrustes
Flow as stated in Algorithm~\ref{alg:rpf} yields a solution $\mtx{U}_0, \mtx{V}_0$ obeying
\begin{align}
\label{initthm_general}
\emph{dist}\left(\blockvec{\mtx{U}_0}{\mtx{V}_0}, \blockvec{\mtx{X}}{\mtx{Y}} \right)\le \frac{1}{4}\sigma_r(\mtx{X}).
\end{align}
Furthermore, take a constant step size $\mu_\tau=\mu$ for all $\tau=1,2,\ldots$
and assume $\mu \le 2/187$.  Then, starting from any initial solution obeying
\eqref{initthm_general}, the $\tau$-th iterate of Algorithm~\ref{alg:rpf} satisfies
\begin{align}
\label{convthm_general}
\emph{dist}\left(\blockvec{\mtx{U}_\tau}{\mtx{V}_\tau},\blockvec{\mtx{X}}{\mtx{Y}}\right)\le \frac{1}{4}\left(1-  \frac{4}{25} \frac{\mu}{\kappa}\right)^{\frac{\tau}{2}}\sigma_r(\mtx{X}).
\end{align}
\end{theorem}

The above theorem shows that Procrustes Flow algorithm achieves a good
initialization under the RIP assumptions on the mapping
$\mathcal{A}$. Also, starting from any sufficiently accurate initialization the
algorithm exhibits geometric convergence to the unknown matrix $\mtx{M}$. We
note that in the above result we have not attempted to optimize the constants.
Furthermore, there is a natural tradeoff involved between the upper bound on
the RIP constant, the radius in which PF is contractive \eqref{initthm_general}, and
its rate of convergence \eqref{convthm_general}. In particular, as it will become clear
in the proofs one can increase the radius in which PF is contractive (increase
the constant $1/4$ in \eqref{initthm_general}) and the rate of convergence (increase
the constant $4/25$ in \eqref{convthm_general}) by assuming a smaller upper bound on
the RIP constant.

The most common measurement ensemble which satisfies the isotropy and RIP
assumptions is the Gaussian ensemble here each matrix $\mtx{A}_k$ has i.i.d.
$\mathcal{N}(0,1/m)$ entries.\footnote{We note that in the PSD case the so
called \emph{spiked Gaussian ensemble} would be the right equivalent. In this
case each symmetric matrix $\mtx{A}_k$ has $\Normal(0, 1/m)$ entries on the
diagonal and $\Normal(0, 1/2m)$ entries elsewhere.} For this ensemble to
achieve a RIP constant of $\delta_r$, we require at least $m =
\Omega(\frac{1}{\delta_r^2} nr)$ measurements. Using equation \eqref{convthm_general} together with a simple calculation detailed in Appendix~\ref{sec:appendix:simplecalc}, we can conclude that for $\mtx{M}_\tau=\mtx{U}_\tau\mtx{V}_\tau^\T$, we have
\begin{align}
\label{myineq}
\fronorm{\mtx{M}_\tau-\mtx{M}}  &\le \frac{9}{4} \sqrt{\sigma_1(\mtx{M})}\cdot\text{dist}\left(\blockvec{\mtx{U}_\tau}{\mtx{V}_\tau},\blockvec{\mtx{X}}{\mtx{Y}}\right)\nonumber\\
                                &\le \frac{9}{16} \sqrt{\sigma_1(\mtx{M})\sigma_r(\mtx{M})}\left(1-  \frac{4}{25} \frac{\mu}{\kappa}\right)^{\frac{\tau}{2}}\nonumber\\
                                &\le \frac{9}{16}\fronorm{\mtx{M}}\left(1-  \frac{4}{25} \frac{\mu}{\kappa}\right)^{\frac{\tau}{2}}.
\end{align}
Thus, applying Theorem \ref{mainthm_general} to this measurement ensemble, we
conclude that the Procrustes Flow algorithm yields a solution with relative
error ($\fronorm{\mtx{M}_\tau-\mtx{M}}/\fronorm{\mtx{M}}\le \epsilon$)
in $\BigO(\kappa\log(1/\epsilon))$ iterations using only $\BigOmega(nr)$
measurements. We would like to note that if more measurements are available it
is not necessary to use multiple projected gradient updates in the
initialization phase. In particular, for the Gaussian model if $m = \BigOmega(n
r^2 \kappa^2)$, then \eqref{initthm} will hold after the first iteration
($\tinit=1$).

\paragraph{How to verify the initialization is complete.}
Theorems \ref{mainthm} and \ref{mainthm_general} require that $\tinit =
\BigOmega(\log(\sqrt{r}\kappa))$, but $\kappa$ is a property of $\mtx{M}$ and
is hence unknown.  However, under the same hypotheses regarding the RIP
constant in Theorems \ref{mainthm} and \ref{mainthm_general}, we can use each
iterate of initialization to test whether or not we have entered the radius of
convergence.
The following lemma
establishes a sufficient condition we can check using only information from
$\widetilde{\mtx{M}}_\tau$.  We establish this result only in the symmetric case-- the
extension to the general case is straightforward. The proof is deferred to
Appendix~\ref{proofstopinit}.
\begin{lemma}
\label{lemma:stopinit}
Assume the RIP constant of $\A$ satisfies $\delta_{2r} \leq 1/10$. Let $\widetilde{\mtx{M}}_\tau$ denote the
$\tau$-th step of the initialization phase in Algorithm~\ref{alg:pf}, and let $\mtx{U}_0 \in \R^{n \times r}$ be the
such that $\widetilde{\mtx{M}}_\tau = \mtx{U}_0 \mtx{U}_0^\T$. Define
\begin{align*}
  e_\tau := \twonorm{ \A(\widetilde{\mtx{M}}_\tau) - \vct{b} } = \twonorm{ \A(\widetilde{\mtx{M}}_\tau - \XXT) } .
\end{align*}
Then, if
\begin{align*}
  e_\tau \leq \frac{3}{20} \sigma_r(\widetilde{\mtx{M}}_\tau) \:,
\end{align*}
we have that
\begin{align*}
  \dist(\mtx{U}_0, \mtx{X}) \leq \frac{1}{4} \sigma_r(\mtx{X}) \:.
\end{align*}
\end{lemma}
One might consider using solely the projected gradient updates (i.e.
set $\tinit=\infty$) as in previous approaches \cite{OMP2, garg2009gradient,
CoSamp, needell2009uniform, blumensath2009, meka09, cai2010singular}. We note
that the projected gradient updates in the initialization phase require
computing the first $r$ singular vectors of a matrix whereas the gradient
updates do not require any singular vector computations. Such singular
computations may be prohibitive compared to the gradient updates, especially when
$n_1$ or $n_2$ is large and for ensembles where matrix-vector multiplication is fast.  We
would like to emphasize, however, that for small $n_1, n_2$ and dense matrices using
projected gradient updates may be more efficient.
Our scheme is a natural interpolation: one could only do projected gradient
steps, or one could do one projected gradient step.  Here we argue that very
few projected gradients provide sufficient initialization such that gradient
descent converges geometrically.

\section{Related work}
\label{sec:related}

There is a vast literature dedicated to low-rank matrix recovery/sensing and
semidefinite programming. We shall only focus on the papers most related to our
framework.

Recht, Fazel, and Parrilo were the first to study low-rank solutions of linear
matrix equations under RIP assumptions~\cite{recht10}. They showed that if the
rank-$r$ RIP constant of $\mathcal{A}$ is less than a fixed numerical constant,
then the matrix with minimum trace satisfying the equality constraints
coincided with the minimum rank solution. In particular, for the Gaussian
ensemble the required number of measurements is
$\BigOmega(nr)$~\cite{CandesPlanTight}.  Subsequently, a series of papers
\cite{candes2009exact, gross2011recovering, recht2011simpler, candes2014phase}
showed that trace minimization and related convex optimization approaches also
work for other measurement ensembles such as those arising in matrix completion
and related problems. In this paper we have established a similar result to
\cite{recht10}. We require the same order of
measurements $\BigOmega(nr)$ but use a more computationally friendly local
search algorithm. Also related to this work are projection gradient schemes
with hard thresholding \cite{OMP2, garg2009gradient,CoSamp,
needell2009uniform,blumensath2009, meka09, cai2010singular}.  Such algorithms
enjoy similar guarantees to that of \cite{recht10} and this work. Indeed, we
utilize such results in the initialization phase of our algorithm. However,
such algorithms require a rank-$r$ SVD in each iteration which may be expensive
for large problem sizes. We would like to emphasize, however, that for small
problem sizes and dense matrices (such as Gaussian ensembles) such algorithms
may be faster than gradient descent approaches such as ours.

More recently, there has been a few results using non-convex optimization
schemes for matrix recovery problems. In particular, theoretical guarantees for
matrix completion have been established using manifold optimization
\cite{keshavan2010matrix} and alternating minimization
\cite{keshavan2012efficient} (albeit with the caveat of requiring a fresh set
of samples in each iteration).  See also \cite{hardt14a,sun14}.
Later on, Jain et.al.~\cite{jain12} analyzed the performance of alternating
minimization under similar modeling assumptions to \cite{recht10} and this
paper. However, the requirements on the RIP constant in \cite{jain12} are more
stringent compared to \cite{recht10} and ours. In particular, the authors
require $\delta_{4r}\le c/r$ whereas we only require $\delta_{6r}\le c$.
Specialized to the Gaussian model, the results of \cite{jain12} require
$\BigOmega(nr^3\kappa^2)$ measurements.\footnote{The authors also propose a
stage-wise algorithm with improved sample complexity of
$\BigOmega(nr^3\tilde{\kappa}^2)$ where $\tilde{\kappa}$ is a local condition
number defined as the ratio of the maximum ratio of two successive eigenvalues.
We note, however, that in general $\tilde{\kappa}$ can be as large as
$\kappa$.} 

Our algorithm and analysis are inspired by the recent paper \cite{candes14} by
Candes, Li and Soltanolkotabi. See also \cite{soltanolkotabi2014algorithms,
cai2015optimal} for some stability results.  In \cite{candes14} the authors
introduced a local regularity condition to analyze the convergence of a
gradient descent-like scheme for phase retrieval.  We use a similar regularity
condition but generalize it to ranks higher than one.  Recently, independent of
our work, Zheng and Lafferty \cite{zheng15} provided an analysis of gradient
descent using (\ref{eq:theproblem}) via the same regularity condition.  Zheng
and Lafferty focus on the Gaussian ensemble, and establish a sample complexity
of $m = \BigOmega( n r^3 \kappa^2 \log n)$. In
comparison we only require $\BigOmega(nr)$ measurements removing both the
dependence on $\kappa$ in the sample complexity and improving the asymptotic
rate. We would like to emphasize that the improvement in our result is not just
due to the more sophisticated initialization scheme. In particular, Zheng and
Lafferty show geometric convergence starting from any
initial solution obeying dist$\left(\mtx{U}_0,\mtx{X}\right)\le
c\cdot \sigma_r(\mtx{X})$ as long as the number of measurements obeys
$m=\BigOmega(nr\kappa^2\log n)$. In contrast, we establish geometric
convergence starting from the same neighborhood of $\mtx{U}_0$ with only
$\BigOmega(nr)$ measurements.  Our results also differs in terms of the convergence rate. We establish a convergence rate of the form $1-\frac{\mu}{\kappa}$ whereas \cite{zheng15} establishes a slower convergence rate of the form $1-\frac{\mu}{nr^2\kappa^2}$. Moreover, the theory of restricted isometries in our work considerably simplifies the analysis.

Finally, we would also like to mention \cite{desa14} for guarantees using
stochastic gradient algorithms. The results of \cite{desa14} are applicable to a
variety of models; focusing on the Gaussian ensemble, the authors require
$\BigOmega\left((nr\log n)/\epsilon\right)$ samples to reach a relative
error of $\epsilon$. In contrast, our sample complexity is independent of the
desired relative error $\epsilon$. However, their algorithm only requires
a random initialization.

%
Since the first version of this paper appeared on arXiv, a few recent papers
have also studied low-rank recovery from RIP measurements via Procrustes Flow
type schemes  \cite{bhojanapalli15,zhao15,chen15}. We would like to point out
that the results presented in these papers are suboptimal compared to ours. For
example,  by utilizing some of the results of the previous version of this
paper, \cite{bhojanapalli15} provides a similar convergence rate to ours.
However, this convergence occurs in a smaller radius around the planted
solution so that the required number of measurements is significantly higher.
Furthermore, the results of \cite{bhojanapalli15} only apply when the matrix is
PSD and do not work for general rectangular. Similarly,
result in \cite{chen15} holds only for PSD matrices, and the convergence rate
has a high-degree polynomial dependence on condition number. The algorithm from
\cite{chen15} does generalize to rectangular matrices, but the sample
complexity is of the order of $\mathcal{O}(nr^3 \log n)$ rather than the
complexity $\mathcal{O}(nr)$ we establish here. Moreover, our analysis of both
the PSD and rectangular cases is far more concise.

\section{Proofs}
\label{sec:proofs}

We first prove our results for the symmetric PSD case (Theorem \ref{mainthm}). However, whenever possible we will
state lemmas in the more general setting. The changes required for the proof of the general setting (Theorem \ref{mainthm_general}) is deferred to Section \ref{secgeneral}.

Recall in this setting that we assume a fixed symmetric PSD $\mtx{M} \in \R^{n \times n}$ of rank $r$,
which admits a factorization $\mtx{M} = \XXT$ for $\mtx{X} \in \R^{n \times r}$.
Before we dive into the details of the proofs, we would like to mention that we will prove our results using the update
\begin{align}
\label{act}
\mtx{U}_{\tau+1}=\mtx{U}_\tau-\frac{\mu}{\opnorm{\mtx{X}}^2}\nabla f(\mtx{U}_\tau),
\end{align}
in lieu of the PF update
\begin{align}
\label{sortof}
\mtx{U}_{\tau+1}=\mtx{U}_\tau-\frac{\mu_{\mathrm{PF}}}{\opnorm{\mtx{U}_0}^2}\nabla f(\mtx{U}_\tau).
\end{align}
As we prove in Section~\ref{sec:proofs:initialization}, our initial
solution obeys $\dist(\mtx{U}_0, \mtx{X}) \leq
\sigma_r(\mtx{X})/4$. Hence, applying triangle inequality we can can conclude that
\begin{align}
  \norm{\mtx{U}_0}^2 \leq \frac{25}{16} \norm{\mtx{X}}^2, \label{eq:scaleone}
\end{align}
and similarly,
\begin{align}
  \norm{\mtx{U}_0}^2 \geq \frac{9}{16} \norm{\mtx{X}}^2 \:. \label{eq:scaletwo}
\end{align}
Thus, any result proven for the update \eqref{act} will automatically carry
over to the PF update with a simple rescaling of the upper bound on the step
size via \eqref{eq:scaleone}. Furthermore, we can upper bound the convergence
rate of gradient descent using the PF update in terms of properties of
$\mtx{X}$ instead of $\mtx{U}_0$ via \eqref{eq:scaletwo}.

\subsection{Preliminaries}
We start with a well known characterization of RIP.
\begin{lemma}
\label{lemma:rip:innerproducts}
\cite{candes08} Let $\A$ satisfy $2r$-RIP with constant $\delta_{2r}$. Then,
for all matrices $\mtx{X}, \mtx{Y}$ of rank at most $r$, we have
\beqs
    \abs{\ip{\A(\mtx{X})}{\A(\mtx{Y})} - \ip{\mtx{X}}{\mtx{Y}}} \leq \delta_{2r} \fronorm{\mtx{X}} \fronorm{\mtx{Y}} \:.
\eeqs
\end{lemma}
Next, we state a recent result which characterizes the convergence rate of projected gradient descent onto general non-convex sets specialized to our problem. See
\cite{meka09} for related results using singular value hard thresholding. Throughout, $\mathcal{P}_r(\mtx{M})$ denotes projection onto rank-$r$ matrices. For a symmetric PSD matrix $\mtx{M}\in\R^{n\times n}$ denotes projection onto the rank-$r$ PSD matrices and for a rectangular matrix $\mtx{M} \in \R^{n_1\times n_2}$ it denotes projection onto rank-$r$ matrices.
\begin{lemma}
\label{lemma:hardthresholding}
\cite{oymak15}
Let $\mtx{M} \in \R^{n_1 \times n_2}$ be an arbitrary matrix of rank $r$.
Also let $\vct{b}=\A(\mtx{M})\in\R^m$ be $m$ linear measurements.
Consider the iterative updates
\beqs
    \mtx{Z}_{\tau+1} \gets \mathcal{P}_{r}\left(\mtx{Z}_\tau - \frac{1}{m} \sum_{k=1}^{m} (\ip{\mtx{A}_k}{\mtx{Z}_\tau} - b_k) \mtx{A}_k\right) \:.
\eeqs
Then
\beqs
    \fronorm{ \mtx{Z}_\tau - \mtx{M} } \leq \rho(\A)^{\tau} \fronorm{ \mtx{Z}_0 - \mtx{M} },
\eeqs
holds. Here, $\rho(\A)$ is defined as
\beqs
    \rho(\A) := 2 \sup_{\substack{ \fronorm{\mtx{X}} = 1, \rank(\mtx{X}) \leq 2r, \\ \fronorm{\mtx{Y}}=1, \rank(\mtx{Y}) \leq 2r }} \abs{\ip{\A(\mtx{X})}{\A(\mtx{Y})} - \ip{\mtx{X}}{\mtx{Y}}} \:.
\eeqs
\end{lemma}
We shall make repeated use of the following lemma which upper bounds $\fronorm{
\UUTMinusXXT }$ by some factor of $\dist(\mtx{U}, \mtx{X})$
\COND{.}{, which is immediate from two applications of the triangle inequality.}
\begin{lemma}\label{simplem} For any $\mtx{U}\in\R^{n\times r}$ obeying $\dist(\mtx{U},\mtx{X}) \le \frac{1}{4}\opnorm{\mtx{X}}$, we have
\begin{align*}
 \fronorm{\UUTMinusXXT} \le\frac{9}{4} \opnorm{\mtx{X}}\dist(\mtx{U}, \mtx{X}).
\end{align*}
\end{lemma}
\COND{
\begin{proof}
\begin{align*}
  \fronorm{\UUTMinusXXT} &= \fronorm{ \mtx{U}(\UMinusXR)^\T + (\UMinusXR)(\mtx{XR})^T } \\
                         &\leq (\opnorm{\mtx{U}} + \opnorm{\mtx{X}}) \fronorm{\UMinusXR} \\
                         &\leq \frac{9}{4} \opnorm{\mtx{X}}\fronorm{\UMinusXR} \:.
\end{align*}
\end{proof}
}{}
Finally, we also need the following lemma which upper bounds $\dist(\mtx{U},
\mtx{X})$ by some factor of $\fronorm{ \UUTMinusXXT }$. We defer the proof of this
result to Appendix \ref{proofstdlem}.
\begin{lemma}
\label{lemma:dist:upperbound}
For any $\mtx{U}, \mtx{X} \in \R^{n \times r}$, we have
\beqs
    \dist^2(\mtx{U}, \mtx{X}) \leq \frac{1}{2(\sqrt{2}-1) \sigma_r^2(\mtx{X})} \fronorm{ \UUTMinusXXT }^2 \:.
\eeqs
\end{lemma}
We would like to point out that the dependence on $\sigma_r^2(\mtx{X})$ in the lemma above is unavoidable.

\subsection{Proof of convergence of gradient descent updates (Equation \eqref{convthm})}
We first outline the general proof strategy. See
Sections 2.3 and 7.9 of \cite{candes14}
for related arguments. We first will show that gradient
descent on an approximate estimate of the function $f$ converges. The approximate function we use is $F(\mtx{U}) := \frac{1}{4} \fronorm{\UUTMinusXXT}^2$. When the map $\A$ is random and isotropic in expectation, $F(\mtx{U})$ can be interpreted as the expected value of $f(\mtx{U})$, but
we stress that our result is a purely deterministic result. We demonstrate that $F(\mtx{U})$ exhibits geometric convergence in a small neighborhood around $\mtx{X}$.
The standard approach in optimization to show this is to prove that
the function exhibits strong convexity.  However, due to the rotational degrees of freedom for any optimal point, it is not possible for
$F(\mtx{U})$ to be strongly convex in any neighborhood around
$\mtx{X}$ except in the special case when $r=1$.
Thus, we rely on the approach used by \cite{candes14}, which establishes a sufficient
condition that only relies on first-order information along certain trajectories.
After showing the sufficient condition holds on $F(\mtx{U})$, we use standard RIP results
to show that this condition also holds for the function $f(\mtx{U})$.

To begin our analysis, we start with the following formulas for the gradient of $f(\mtx{U})$ and $F(\mtx{U})$
\COND{
\beqs
\nabla f(\mtx{U}) = \sum_{k=1}^{m} \ip{\mtx{A}_k}{\UUTMinusXXT} \mtx{A}_k \mtx{U} = \A^*\A(\UUTMinusXXT) \cdot \mtx{U}, \:\: \nabla F(\mtx{U}) = (\UUTMinusXXT) \mtx{U} \:.
\eeqs
}{
\begin{align*}
    \nabla f(\mtx{U}) &= \sum_{k=1}^{m} \ip{\mtx{A}_k}{\UUTMinusXXT} \mtx{A}_k \mtx{U} \\
                      &= \A^*\A(\UUTMinusXXT) \cdot \mtx{U} \:, \\
    \nabla F(\mtx{U}) &= (\UUTMinusXXT) \mtx{U} \:.
\end{align*}
}
Above, $\A^* : \R^{m} \rightarrow \R^{n \times n}$ is the adjoint operator of $\A$, i.e. $\A^*(z) = \sum_{i=1}^{m} \mtx{A}_k z_k$.
Throughout the proof $\mtx{R} $ is the solution to the orthogonal Procrustes problem. That is,
\begin{align*}
  \mtx{R}=\underset{\widetilde{\mtx{R}}\in\R^{n\times n}:\text{ }\widetilde{\mtx{R}}^\T\widetilde{\mtx{R}}=\mtx{I}_r}{\arg\min}\text{ }\fronorm{\mtx{U}-\mtx{X}\widetilde{\mtx{R}}},
\end{align*}
with the dependence on $\mtx{U}$ omitted for sake of exposition. The following definition defines a notion of strong convexity along
certain trajectories of the function.
\begin{definition}
\label{def:rc}
(Regularity condition, \cite{candes14})
Let $\mtx{X}\in\R^{n\times r}$ be a global optimum of a function $f$. Define the set $B(\delta)$ as
\begin{align*}
    B(\delta) := \{ \mtx{U} \in \R^{n \times r} : \dist(\mtx{U}, \mtx{X}) \leq \delta \} \:.
\end{align*}
The function $f$ satisfies a \emph{regularity condition}, denoted by $\RC(\alpha,
\beta, \delta)$, if for all matrices $\mtx{U} \in
B(\delta)$ the following inequality holds:
\begin{align*}
    \ip{\nabla f(\mtx{U})}{\UMinusXR} \geq \frac{1}{\alpha} \fronorm{\UMinusXR}^2 + \frac{1}{\beta} \fronorm{\nabla f(\mtx{U})}^2 \:.
\end{align*}
\end{definition}
If a function satisfies $\RC(\alpha,\beta,\delta)$, then as long as gradient
descent starts from a point $\mtx{U}_0 \in B(\delta)$, it will have a geometric rate of
convergence to the optimum $\mtx{X}$. This is formalized by the following lemma.
\begin{lemma}
\label{lemma:rc_implies_convergence}
\cite{candes14} If $f$ satisfies $\RC(\alpha,\beta,\delta)$ and $\mtx{U}_0 \in
B(\delta)$, then the gradient descent update
\begin{align*}
    \mtx{U}_{\tau+1} \gets \mtx{U}_\tau - \mu \nabla f(\mtx{U}_\tau),
\end{align*}
with step size $0 < \mu \leq 2/\beta$ obeys
$\mtx{U}_\tau \in B(\delta)$ and
\begin{align*}
    \dist^2(\mtx{U}_\tau, \mtx{X}) \leq \left(1-\frac{2\mu}{\alpha}\right)^{\tau} \dist^2(\mtx{U}_0, \mtx{X}) \:,
\end{align*}
for all $\tau \geq 0$.
\end{lemma}

The proof is complete by showing that the regularity condition holds. To this
end, we first show in Lemma \ref{lemma:expected_f_rc_cond} below that the
function $F(\mtx{U})$ satisfies a slightly stronger variant of the
regularity condition from Definition~\ref{def:rc}. We then show in Lemma
\ref{lemma:concentration:gradient} that the gradient of $f$ is always close to
the gradient of $F$, and in Lemma \ref{lemma:concentration:gradientnorm} that
the gradient of $f$ is Lipschitz around the optimal value $\mtx{X}$.

\begin{lemma}
\label{lemma:expected_f_rc_cond}
Let $F(\mtx{U})=\frac{1}{4}\fronorm{\UUTMinusXXT}^2$. For all $\mtx{U}$ obeying
\begin{align*}
\opnorm{\mtx{U}-\mtx{X}\mtx{R}}\le \frac{1}{4}\sigma_r(\mtx{X}),
\end{align*}
we have
\begin{align}
\COND{
\ip{\nabla F(\mtx{U})&}{\UMinusXR} - \frac{1}{20}\left( \fronorm{\UUTMinusXXT}^2 + \fronorm{(\UMinusXR)\mtx{U}^\T}^2 \right) \nonumber \\
                           &\geq \frac{\sigma_r^2(\mtx{X})}{4} \fronorm{\UMinusXR}^2 + \frac{1}{5}\fronorm{\UUTMinusXXT}^2 \label{mahdilemeq}.
}{
  \ip{\nabla F(\mtx{U})&}{\UMinusXR} \geq \nonumber \\
                       &\frac{1}{20}\left( \fronorm{\UUTMinusXXT}^2 + \fronorm{(\UMinusXR)\mtx{U}^\T}^2 \right) \nonumber \\
                       &+ \frac{\sigma_r^2(\mtx{X})}{4} \fronorm{\UMinusXR}^2 + \frac{1}{5}\fronorm{\UUTMinusXXT}^2 \label{mahdilemeq}.
}
\end{align}
\end{lemma}
\begin{lemma}
\label{lemma:concentration:gradient}
Let $\mathcal{A}$ be a linear map obeying rank-4r RIP with constant
$\delta_{4r}$. For any $\mtx{H} \in \R^{n \times r}$ and any $\mtx{U} \in \R^{n
\times r}$ obeying $\dist(\mtx{U},\mtx{X})\le \frac{1}{4}
\opnorm{\mtx{X}}$, we have
\beqs
    \abs{ \ip{ \nabla F(\mtx{U}) - \nabla f(\mtx{U}) }{ \mtx{H} } } \leq \delta_{4r} \fronorm{\UUTMinusXXT} \fronorm{\mtx{H}\mtx{U}^\T} \:.
\eeqs
This immediately implies that for any $\mtx{U} \in \R^{n \times r}$ obeying $\dist(\mtx{U},\mtx{X})\le \frac{1}{4}\opnorm{\mtx{X}}$, we have
\beqs
    \fronorm{ \nabla f(\mtx{U}) - \nabla F(\mtx{U}) } \leq \delta_{4r} \fronorm{ \UUTMinusXXT } \opnorm{\mtx{U}} \:.
\eeqs
\end{lemma}

\begin{lemma}
\label{lemma:concentration:gradientnorm}
Let $\mathcal{A}$ be a linear map obeying rank-$6r$ RIP with constant $\delta_{6r}$. Suppose that
$\delta_{6r} \leq 1/10$. Then for all $\mtx{U} \in \R^{n \times r}$, we have that
\beqs
    \fronorm{\UUTMinusXXT}^2 \geq \frac{10}{17} \frac{1}{\norm{\mtx{U}}^2} \fronorm{ \nabla f(\mtx{U}) }^2 \:.
\eeqs
\end{lemma}

We shall prove these three lemmas in Sections \ref{sec:proofs:first:regularity},
\ref{secconcengrad}, and \ref{secconcengradnorm}. However, we first explain how
the regularity condition follows from these three lemmas. To begin, note that
\begin{align}
\COND{
  \ip{\nabla F(\mtx{U})}{\UMinusXR} &= \ip{\nabla f(\mtx{U})}{\UMinusXR} + \ip{\nabla F(\mtx{U}) - \nabla f(\mtx{U})}{\UMinusXR} \nonumber \\
                                    &\stackrel{(a)}{\leq}  \ip{\nabla f(\mtx{U})}{\UMinusXR} + \frac{1}{10} \fronorm{\UUTMinusXXT} \fronorm{(\UMinusXR)\mtx{U}^\T} \nonumber \\
                                    &\stackrel{(b)}{\leq}  \ip{\nabla f(\mtx{U})}{\UMinusXR} + \frac{1}{20} \left( \fronorm{\UUTMinusXXT}^2 + \fronorm{(\UMinusXR)\mtx{U}^\T}^2 \right)\label{eq:grad_f_upper_bound}
}{
    \ip{\nabla &F(\mtx{U})}{\UMinusXR}- \ip{\nabla f(\mtx{U})}{\UMinusXR}\\
                                    &= \ip{\nabla F(\mtx{U}) - \nabla f(\mtx{U})}{\UMinusXR} \nonumber \\
                                    &\stackrel{(a)}{\leq}  \frac{1}{10} \fronorm{\UUTMinusXXT} \fronorm{(\UMinusXR)\mtx{U}^\T} \nonumber \\
                                    &\stackrel{(b)}{\leq}  \frac{1}{20} \left( \fronorm{\UUTMinusXXT}^2 + \fronorm{(\UMinusXR)\mtx{U}^\T}^2 \right)\label{eq:grad_f_upper_bound}
}
\end{align}
where (a) holds from Cauchy-Schwarz followed by
Lemma~\ref{lemma:concentration:gradient}, using the fact that $\delta_{6r}\le
\frac{1}{10}$ as assumed in the statement of Theorem \ref{mainthm} and (b)
follows from $2ab \leq a^2 + b^2$.

Combining \eqref{eq:grad_f_upper_bound} with Lemma \ref{lemma:expected_f_rc_cond} for any $\mtx{U}$ obeying $\opnorm{\UMinusXR}\le \frac{1}{4}\sigma_r(\mtx{X})$, we have
\begin{align}
    \ip{\nabla \COND{}{&}f(\mtx{U})}{\UMinusXR} \COND{&}{\nonumber \\ &}\geq \frac{\sigma_r^2(\mtx{X})}{4} \fronorm{\UMinusXR}^2 + \frac{1}{5} \fronorm{\UUTMinusXXT}^2 \nonumber\\
                                      &\stackrel{(a)}{\geq} \frac{\sigma_r^2(\mtx{X})}{4} \fronorm{\UMinusXR}^2 + \frac{2}{17} \frac{1}{\norm{\mtx{U}}^2} \fronorm{ \nabla f(\mtx{U}) }^2 \nonumber\\
                                      &\stackrel{(b)}{\geq} \frac{\sigma_r^2(\mtx{X})}{4} \fronorm{\UMinusXR}^2 + \frac{32}{425} \frac{1}{\norm{\mtx{X}}^2} \fronorm{ \nabla f(\mtx{U}) }^2 \label{intermahdi} \:,
\end{align}
where (a) follows from Lemma~\ref{lemma:concentration:gradientnorm} and
(b) follows from the fact that $\norm{\mtx{U}} \leq \frac{5}{4} \norm{\mtx{X}}$
when $\dist(\mtx{U},\mtx{X})\le \frac{1}{4}\opnorm{\mtx{X}}$.
Equation~\eqref{intermahdi} shows that $f(\mtx{U})$ obeys $\RC(4/\sigma_r^2(\mtx{X}),
\frac{425}{32} \opnorm{\mtx{X}}^2, \frac{1}{4} \sigma_r(\mtx{X}))$. The
convergence result in Equation~\eqref{convthm} now follows from
Lemma~\ref{lemma:rc_implies_convergence}.  \COND{All that remains is to prove Lemmas
\ref{lemma:expected_f_rc_cond}, \ref{lemma:concentration:gradient},
and \ref{lemma:concentration:gradientnorm}.}{}

\if\MODE1
\COND{}{\section{Missing proofs for PSD case from Section~\ref{sec:proofs}}}

\SUPsubsection{Proof of the regularity condition for the function $F$ (Lemma~\ref{lemma:expected_f_rc_cond})}
\label{sec:proofs:first:regularity}
We first state some properties of the Procrustes problem and its optimal
solution. Let $\mtx{U}, \mtx{X} \in \R^{n \times r}$ and define $\mtx{H} :=
\UMinusXR$, where $\mtx{R}$ is the orthogonal matrix which minimizes
$\fronorm{\mtx{U}-\mtx{X}\mtx{R}}$. Let $\mtx{A} \mtx{\Sigma} \mtx{B}^\T$ be
the SVD of $\mtx{X}^\T\mtx{U}$; we know that the optimal $\mtx{R}$ is $\mtx{R} = \mtx{A}\mtx{B}^\T$. Thus,
\beqs
\mtx{U}^\T \mtx{XR} = \mtx{B} \mtx{\Sigma} \mtx{B}^\T = (\mtx{XR})^\T \mtx{U},
\eeqs
which shows that $\mtx{U}^\T\mtx{XR}$ is a symmetric PSD matrix. Furthermore, note that since
\beqs
\mtx{H}^\T \mtx{XR} = \mtx{U}^\T \mtx{XR} - \mtx{R}^\T \mtx{X}^\T \mtx{X R} = (\mtx{XR})^\T \mtx{ U } - \mtx{R}^\T \mtx{X}^\T \mtx{X R} = (\mtx{XR})^\T (\mtx{U} - \mtx{XR}) = (\mtx{XR})^\T \mtx{H} \:,
\eeqs
we can conclude that $\mtx{H}^\T \mtx{XR}$ is symmetric. To avoid carrying $\mtx{R}$ in our equations we perform the change of variable $\mtx{X}\leftarrow\mtx{X}\mtx{R}$. That is, without loss of generality we assume $\mtx{R}=\mtx{I}$ and that $\mtx{U}^\T \mtx{X} \succeq 0$ and $\mtx{H}^\T \mtx{X} = \mtx{X}^\T \mtx{H}$.

Note that for any $\mtx{U}$ obeying dist$(\mtx{U},\mtx{X})\le \frac{1}{4}\opnorm{\mtx{X}}$ we have
\begin{align*}
  \fronorm{(\UUTMinusXXT)\mtx{U}} \leq \fronorm{\UUTMinusXXT}\opnorm{\mtx{U}} \leq \frac{5}{4} \opnorm{\mtx{X}} \fronorm{\UUTMinusXXT}.
\end{align*}
Using the latter along with the simplifications discussed above, to prove Lemma~\eqref{mahdilemeq} it suffices to prove
\begin{align}
\ip{(\UUTMinusXXT)\mtx{U}&}{\mtx{U}-\mtx{X}} - \frac{1}{20}\left( \fronorm{\UUTMinusXXT}^2 + \fronorm{(\mtx{U}-\mtx{X})\mtx{U}^\T}^2 \right) \nonumber \\
                         &\geq \frac{\sigma_r^2(\mtx{X})}{4} \fronorm{\mtx{U}-\mtx{X}}^2 + \frac{1}{5} \fronorm{\UUTMinusXXT}^2. \label{eq:almost_rc_easier}
\end{align}
Equation \eqref{eq:almost_rc_easier} can equivalently be written in the form
\begin{align*}
  0 \leq \Tr\bigg(
  &(\mtx{H}^\T\mtx{H})^2 + 3 \mtx{H}^\T\mtx{H} \mtx{H}^\T\mtx{X} + (\mtx{H}^\T\mtx{X})^2 + \mtx{H}^\T\mtx{H} \mtx{X}^\T\mtx{X} \\
  &- \left( \frac{1}{20} + \frac{1}{5} \right)\left[  (\mtx{H}^\T\mtx{H})^2 + 4 \mtx{H}^\T\mtx{H} \mtx{H}^\T\mtx{X} + 2 (\mtx{H}^\T\mtx{X})^2 + 2 \mtx{H}^\T\mtx{H} \mtx{X}^\T\mtx{X}  \right] \\
  &- \frac{1}{20} \left[  (\mtx{H}^\T\mtx{H})^2 + 2 \mtx{H}^\T\mtx{H} \mtx{H}^\T\mtx{X} + \mtx{H}^\T\mtx{H} \mtx{X}^\T\mtx{X}  \right] \\
  &-\frac{\sigma_r^2(\mtx{X})}{4} \mtx{H}^\T \mtx{H}\bigg) \:.
\end{align*}
Rearranging terms, we arrive at
\begin{align}
  0 &\leq \Tr\bigg( c_1 (\mtx{H}^\T\mtx{H})^2 + c_2 \mtx{H}^\T\mtx{H} \mtx{H}^\T\mtx{X} + c_3 (\mtx{H}^\T\mtx{X})^2 + c_4 \mtx{H}^\T\mtx{H} \mtx{X}^\T\mtx{X} - \frac{\sigma_r^2(\mtx{X})}{4} \mtx{H}^\T\mtx{H} \bigg) \nonumber \\
    &= \Tr\bigg(  \big(\frac{c_2}{2\sqrt{c_3}} \mtx{H}^\T\mtx{H} + \sqrt{c_3} \mtx{H}^\T\mtx{X}\big)^2 + \big(c_1 - \frac{c_2^2}{4 c_3}\big) (\mtx{H}^\T\mtx{H})^2 + c_4 \mtx{H}^\T\mtx{H} \mtx{X}^\T\mtx{X} - \frac{\sigma_r^2(\mtx{X})}{4} \mtx{H}^\T\mtx{H}\bigg). \label{eq:radius_sufficient}
\end{align}
Here the constants $c_1,c_2,c_3,c_4$ are defined as
\begin{align*}
  c_1 = \frac{7}{10}, \;\;
  c_2 = \frac{19}{10}, \;\;
  c_3 = \frac{1}{2}, \;\;
  c_4 = \frac{9}{20} \:.
\end{align*}
Since $c_1 < \frac{c_2^2}{4c_3}$, to prove \eqref{eq:radius_sufficient} it thus suffices to require that
\begin{align*}
    \opnorm{ \mtx{H} }^2 \leq \frac{c_4-1/4}{ c_2^2/4c_3 - c_1} \sigma_r^2(\mtx{X}) = \frac{40}{221} \sigma_r^2(\mtx{X}) \:.
\end{align*}

\SUPsubsection{Proof of gradient concentration (Lemma~\ref{lemma:concentration:gradient})}
\label{secconcengrad}
Define $\mtx{\Delta} := \UUTMinusXXT$. Then,
\begin{align*}
    \abs{\ip{ \nabla f(\mtx{U}) - \nabla F(\mtx{U}) }{\mtx{H}} } = \abs{ \ip{\A(\mtx{\Delta})}{\A(\mtx{H}\mtx{U}^\T)} - \ip{\mtx{\Delta}}{\mtx{H}\mtx{U}^\T} }
                                                                &\stackrel{(a)}{\leq} \delta_{4r} \fronorm{\UUTMinusXXT} \fronorm{\mtx{H}\mtx{U}^\T}
\end{align*}
where (a) follows from Lemma~\ref{lemma:rip:innerproducts}, since $\rank(\mtx{\Delta}) \leq 2r$
and $\rank(\mtx{H}\mtx{U}^\T) \leq r$.  This proves the first part of the
lemma. To prove the second part, by the variational form of the Frobenius norm,
we have
\begin{align*}
  \fronorm{ \nabla f(\mtx{U}) - \nabla F(\mtx{U}) } &= \sup_{\mtx{H} \in \R^{n \times r}, \fronorm{\mtx{H}} \leq 1} \ip{ \nabla f(\mtx{U}) - \nabla F(\mtx{U}) }{\mtx{H}} \\
                                                    &\leq \delta_{4r} \fronorm{\UUTMinusXXT} \sup_{\mtx{H} \in \R^{n \times r}, \fronorm{\mtx{H}} \leq 1}\fronorm{\mtx{H}\mtx{U}^\T} \:.
\end{align*}
The result now follows from $\fronorm{\mtx{H}\mtx{U}^\T} \leq \fronorm{\mtx{H}} \opnorm{\mtx{U}} \leq \opnorm{\mtx{U}}$.

\SUPsubsection{Proof of Lipschitz gradients around optimal solution (Lemma~\ref{lemma:concentration:gradientnorm})}
\label{secconcengradnorm}

Define $\mtx{\Delta} := \UUTMinusXXT$ and $\delta := \delta_{6r}$. Suppose we show
that
\begin{align}
    \norm{\A(\mtx{\Delta})}_F^2 - \gamma \norm{\nabla f(\mtx{U})}_F^2 \geq \frac{1}{4} \fronorm{\mtx{\Delta}}^2 \:, \label{eq:concengradnorm:b}
\end{align}
where $\gamma := 1/2 \norm{\mtx{U}}^2$.  Then
by $2r$-RIP we have
\begin{align*}
    \frac{1}{4} \fronorm{\mtx{\Delta}}^2 + \frac{1}{2 \norm{\mtx{U}}^2} \norm{\nabla f(\mtx{U})}^2 \leq \fronorm{\A(\mtx{\Delta})}^2 \leq (1+\delta) \fronorm{\mtx{\Delta}}^2 \leq (1+\frac{1}{10})\fronorm{\mtx{\Delta}}^2 \:,
\end{align*}
which yields the claim after rearranging.

We now focus on proving Equation~\eqref{eq:concengradnorm:b}.  Recall $\nabla
f(\mtx{U}) = \A^*\A(\mtx{\Delta}) \cdot \mtx{U}$, which implies $\norm{\nabla
f(\mtx{U})}_F^2 = \ip{\A(\mtx{\Delta})}{\A( \A^*\A(\mtx{\Delta}) \cdot \UUT
)}$.  Using this equality we have
\begin{align}
    \norm{\A(\mtx{\Delta})}_F^2 &- \gamma\norm{\nabla f(\mtx{U})}_F^2  \nonumber \\
                          &= \ip{\A(\mtx{\Delta})}{\A(\mtx{\Delta}) - \gamma \A( \A^*\A(\mtx{\Delta}) \cdot \UUT ) }  \nonumber \\
                                                       &= \ip{\A(\mtx{\Delta})}{\A( \mtx{\Delta} - \gamma  \A^*\A(\mtx{\Delta}) \cdot \UUT ) }  \nonumber \\
                                                       &\stackrel{(a)}{\geq} \ip{\mtx{\Delta}}{ \mtx{\Delta} - \gamma  \A^*\A(\mtx{\Delta}) \cdot \UUT } - \delta \norm{\mtx{\Delta}}_F \norm{  \mtx{\Delta} - \gamma  \A^*\A(\mtx{\Delta}) \cdot \UUT }_F  \nonumber \\
                                                       &= \ip{\mtx{\Delta}}{\mtx{\Delta}} - \gamma \ip{\mtx{\Delta}}{\A^*\A(\mtx{\Delta}) \cdot \UUT} - \delta \norm{\mtx{\Delta}}_F \norm{  \mtx{\Delta} - \gamma  \A^*\A(\mtx{\Delta}) \cdot \UUT }_F  \nonumber \\
                                                       &= \ip{\mtx{\Delta}}{\mtx{\Delta}} - \gamma \ip{\A(\mtx{\Delta} \UUT)}{\A(\mtx{\Delta})} - \delta \norm{\mtx{\Delta}}_F \norm{  \mtx{\Delta} - \gamma  \A^*\A(\mtx{\Delta}) \cdot \UUT }_F  \nonumber \\
                                                       &\stackrel{(b)}{\geq} \ip{\mtx{\Delta}}{\mtx{\Delta}} - \gamma \ip{\mtx{\Delta} \UUT}{\mtx{\Delta}} - \gamma \delta \norm{\mtx{\Delta} \UUT}_F \norm{\mtx{\Delta}}_F - \delta \norm{\mtx{\Delta}}_F \norm{  \mtx{\Delta} - \gamma  \A^*\A(\mtx{\Delta}) \cdot \UUT }_F \:, \label{eq:concengradnorm:c}
\end{align}
where both (a) and (b) hold by Lemma~\ref{lemma:rip:innerproducts} since
$\rank(\mtx{\Delta}) \leq 2r$,
$\rank(\mtx{\Delta} - \gamma  \A^*\A(\mtx{\Delta}) \cdot \UUT ) \leq 3r$,
and $\rank(\mtx{\Delta} \UUT) \leq r$.
We now control $\norm{ \mtx{\Delta} - \gamma \A^*A(\mtx{\Delta}) \cdot \UUT }_F$ from above.
Using the variational form of the Frobenius norm,
\begin{align*}
  \norm{ \mtx{\Delta} - \gamma \A^*A(\mtx{\Delta}) \cdot \UUT }_F &= \sup_{\mtx{V} \in \R^{n\times n}, \fronorm{\mtx{V}} \leq 1} \ip{ \mtx{\Delta} - \gamma \A^*\A(\mtx{\Delta}) \cdot \UUT }{\mtx{V}} \\
                                                      &= \sup_{\mtx{V} \in \R^{n\times n}, \fronorm{\mtx{V}} \leq 1} \ip{\mtx{\Delta}}{\mtx{V}} - \gamma \ip{\A^*\A(\mtx{\Delta}) \cdot \UUT}{\mtx{V}} \\
                                                      &= \sup_{\mtx{V} \in \R^{n\times n}, \fronorm{\mtx{V}} \leq 1} \ip{\mtx{\Delta}}{\mtx{V}} - \gamma \ip{\A^*\A(\mtx{\Delta})}{\mtx{V} \UUT} \\
                                                      &= \sup_{\mtx{V} \in \R^{n\times n}, \fronorm{\mtx{V}} \leq 1} \ip{\mtx{\Delta}}{\mtx{V}} - \gamma \ip{\A(\mtx{\Delta})}{\A(\mtx{V} \UUT)} \\
                                                      &\stackrel{(a)}{\leq} \sup_{\mtx{V} \in \R^{n\times n}, \fronorm{\mtx{V}} \leq 1} \ip{\mtx{\Delta}}{\mtx{V}} - \gamma \ip{\mtx{\Delta}}{\mtx{V}\UUT} + \gamma \delta \norm{\mtx{\Delta}}_F \norm{\mtx{V} \UUT}_F \\
                                                      &= \sup_{\mtx{V} \in \R^{n\times n}, \fronorm{\mtx{V}} \leq 1} \ip{\mtx{\Delta}(I-\gamma \UUT)}{\mtx{V}} + \gamma \delta \norm{\mtx{\Delta}}_F \norm{\mtx{V} \UUT}_F \\
                                                      &\stackrel{(b)}{\leq} \norm{\mtx{\Delta} - \gamma \mtx{\Delta} \UUT}_F + \gamma \delta \norm{\mtx{\Delta}}_F \norm{\UUT} \:,
\end{align*}
where (a) holds again by Lemma~\ref{lemma:rip:innerproducts} since
$\rank(\mtx{V} \UUT) \leq r$, and (b) holds since $\norm{\mtx{V}\UUT}_F \leq
\fronorm{\mtx{V}} \norm{\UUT} \leq \norm{\UUT}$.  Plugging this upper bound into
Equation~\eqref{eq:concengradnorm:c}, we have
\begin{align}
  \norm{\A(\mtx{\Delta})}_F^2 &- \gamma\norm{\nabla f(\mtx{U})}_F^2  \nonumber\\
  &\geq \ip{\mtx{\Delta}}{\mtx{\Delta}} - \gamma \ip{\mtx{\Delta} \UUT}{\mtx{\Delta}} - \gamma \delta \norm{\mtx{\Delta} \UUT}_F \norm{\mtx{\Delta}}_F - \delta \norm{\mtx{\Delta}}_F \norm{\mtx{\Delta} - \gamma \mtx{\Delta} \UUT}_F - \gamma \delta^2 \norm{\mtx{\Delta}}_F^2 \norm{\UUT} \nonumber\\
  &\stackrel{(a)}{\geq} \ip{\mtx{\Delta}}{\mtx{\Delta}} - \gamma \norm{\UUT} \ip{\mtx{\Delta}}{\mtx{\Delta}} - \gamma \delta \norm{\mtx{\Delta}}_F^2 \norm{\UUT} - \delta \norm{\mtx{\Delta}}_F^2 \norm{I - \gamma \UUT} - \gamma \delta^2 \norm{\mtx{\Delta}}_F^2 \norm{\UUT}\label{myeq}.
\end{align}
where (a) holds since $\ip{\mtx{\Delta} \UUT}{\mtx{\Delta}} \leq \norm{\mtx{U}}^2 \ip{\mtx{\Delta}}{\mtx{\Delta}}$.
Using the fact that $\delta \leq 1/10$ \eqref{myeq} implies,
\begin{align}
  \norm{\A(\mtx{\Delta})}_F^2 &- \gamma\norm{\nabla f(\mtx{U})}_F^2 \nonumber \\
                        &\geq \left(1 - \gamma \sigma_1^2(U) - \frac{\gamma}{10}\sigma_1^2(U) - \frac{\gamma}{100} \sigma_1^2(U) - \frac{1}{10} \norm{ I - \gamma \UUT } \right) \norm{\mtx{\Delta}}_F^2 \label{eq:concengradnorm:a} \:.
\end{align}
For $\gamma = \frac{1}{2 \norm{\mtx{U}}^2}$, we have
\begin{align*}
\norm{I - \gamma \UUT}
= \max( \abs{1 - \gamma \sigma_1^2(\mtx{U})}, \: \abs{1 -
\gamma\sigma_r^2(\mtx{U})}) =\max\left(\frac{1}{2}, \: 1-\frac{\sigma_r^2(\mtx{U})}{2\opnorm{\mtx{U}}^2}\right)= 1-\frac{\sigma_r^2(\mtx{U})}{2\opnorm{\mtx{U}}^2} \leq 1.
\end{align*}
Plugging this bound into
Equation~\eqref{eq:concengradnorm:a}, we get
\begin{align*}
    \norm{\A(\mtx{\Delta})}_F^2 - \frac{1}{2\norm{\mtx{U}}^2} \norm{\nabla f(\mtx{U})}_F^2 \geq \left( 1 - \frac{1}{2} - \frac{1}{20} - \frac{1}{200} - \frac{1}{10} \right) \norm{\mtx{\Delta}}_F^2 \geq \frac{1}{4} \fronorm{\mtx{\Delta}}^2 \:.
\end{align*}

\subsection{Proof of initialization (Equation \eqref{initthm})}
\label{sec:proofs:initialization}
Using Lemma~\ref{lemma:rip:innerproducts}, we can conclude that
$\rho(\A)$ from Lemma~\ref{lemma:hardthresholding} is bounded by
$\rho(\A) \leq 2 \delta_{4r} \leq 1/5$. Setting
$\widetilde{\mtx{M}}_0 = \mtx{0}_{n \times n}$ and applying Lemma~\ref{lemma:hardthresholding}
to our initialization iterates, we have that
\begin{align*}
  \fronorm{ \widetilde{\mtx{M}}_\tau - \XXT } \leq (1/5)^{\tau} \fronorm{\XXT} \leq (1/5)^{\tau} \opnorm{\mtx{X}} \fronorm{\mtx{X}} \:.
\end{align*}
From Lemma~\ref{lemma:dist:upperbound}, we have that
\begin{align*}
  \dist(\mtx{U}_0, \mtx{X}) \leq \frac{\sqrt{2}}{\sigma_r(\mtx{X})} \fronorm{ \widetilde{\mtx{M}}_\tau - \XXT } \leq \sqrt{2} (1/5)^{\tau} \sqrt{\kappa} \fronorm{\mtx{X}} \:.
\end{align*}
Hence, if we want the RHS to be upper bounded by $\frac{1}{4}
\sigma_r(\mtx{X})$, we require
\begin{align*}
  \sqrt{2} (1/5)^{\tau} \sqrt{\kappa} \fronorm{\mtx{X}} \leq \frac{1}{4} \sigma_r(\mtx{X}) \Longrightarrow (1/5)^{\tau} \leq \frac{\sigma_r(\mtx{X})}{4 \sqrt{2} \sqrt{\kappa} \fronorm{\mtx{X}}} \:.
\end{align*}
Since $\fronorm{\mtx{X}} \leq \sqrt{r} \opnorm{\mtx{X}}$, it is enough to require that
\begin{align}
    \tau \geq  \log(\sqrt{r}\kappa) +2 \:. \label{stthree}
\end{align}
Similarly, it is easy to check that if $\tau$ satisfies \eqref{stthree}, then
\begin{align*}
  \opnorm{\widetilde{\mtx{M}}_t - \XXT} \leq \fronorm{\widetilde{\mtx{M}}_t - \XXT} \leq \sigma_r^2(\mtx{X})/4 \:,
\end{align*}
also holds.

\COND{}{\section{Expanded details for Section~\ref{sec:proofs:rect}}}

\subsection{Proof for rectangular matrices (Theorem \ref{mainthm_general})}
\label{secgeneral}

\COND{
We now turn our attention to the general case where the matrices are
rectangular.  Recall that in this case, we want to recover a fixed but unknown
rank-$r$ matrix $\mtx{M} \in \R^{n_1 \times n_2}$ from linear measurements.
Assume that $\mtx{M}$ has a singular value decomposition of the form $\mtx{M} =
\mtx{A} \mtx{\Sigma} \mtx{B}^\T$.  Define $\mtx{X} = \mtx{A} \mtx{\Sigma}^{1/2}
\in \R^{n_1 \times r}$ and $\mtx{Y} = \mtx{B} \mtx{\Sigma}^{1/2} \in \R^{n_2
\times r}$.  With this piece of notation the iterates $\mtx{U}_\tau \in \R^{n_1
\times r}, \mtx{V}_\tau \in \R^{n_2 \times r}$ in Algorithm~\ref{alg:rpf} can
be thought of as estimates of $\mtx{X}$ and $\mtx{Y}$. The proof of the
correctness of the initialization phase of Procrustes Flow (Theorem
\ref{mainthm_general}, Equation \eqref{initthm_general}) in the rectangular case
is similar to the PSD case (Theorem \ref{mainthm}, Equation
\eqref{initthm}) and is detailed in Section \ref{secinitgen}. In this section
we shall focus on proving the convergence guarantee provided in Theorem
\ref{mainthm_general}, Equation \eqref{convthm_general}.
}{
In this section we detail the proof for the rectangular case.
We duplicate some of the material in Section~\ref{sec:proofs:rect} for clarity.
}

To simplify exposition we aggregate the pairs of matrices $(\mtx{U},\mtx{V})$, $(\mtx{U}_\tau,\mtx{V}_\tau)$,  $(\mtx{X},\mtx{Y})$, and $(\mtx{X},-\mtx{Y})$ into larger ``lifted" matrices as follows
\begin{align*}
\mtx{W} := \blockvec{\mtx{U}}{\mtx{V}},\quad\mtx{W}_\tau := \blockvec{\mtx{U}_\tau}{\mtx{V}_\tau},\quad\mtx{Z} := \blockvec{\mtx{X}}{\mtx{Y}},\quad\text{and}\quad \widetilde{\mtx{Z}} := \blockvec{\mtx{X}}{-\mtx{Y}}.
\end{align*}
Before we continue further we first record a few simple facts about these new variables which we will utilize multiple times in the sequel. First, note that for $\ell=1,2,...,r$, we have $\sigma_\ell^2(\mtx{Z})=\sigma_\ell^2(\widetilde{\mtx{Z}})=2\sigma_\ell(\mtx{M})$. Also, since $\mtx{X}^\T\mtx{X}=\mtx{Y}^\T\mtx{Y}=\mtx{\Sigma}$, we have $\mtx{Z}^\T\widetilde{\mtx{Z}}=\widetilde{\mtx{Z}}^\T\mtx{Z}=\mtx{0}_{r\times r}$.

To prove Theorem~\ref{mainthm_general}, Equation~\eqref{convthm_general}, we
will demonstrate that the function $g(\mtx{W}):=g(\mtx{U},\mtx{V})$ over the
variable $\mtx{W}$ has similar form to $f(\mtx{U})$ over the variable
$\mtx{U}$. To see this connection clearly, we need a few useful block matrix
operators and definitions. Let $\Sym : \R^{n_1\times n_2} \longrightarrow
\R^{n_1\times n_2}$ be defined as
\begin{align*}
    \Sym(\mtx{A}) := \blockmat{\mtx{0}_{n_1 \times n_1}}{\mtx{A}}{\mtx{A}^\T}{\mtx{0}_{n_2 \times n_2}} \:.
\end{align*}
We note for future use that with this notation we have $\Sym(\mtx{M}) = \frac{1}{2}(\mtx{Z}\mtx{Z}^\T - \widetilde{\mtx{Z}}\widetilde{\mtx{Z}}^\T)$.
Given a block matrix $\mtx{A} \in \R^{(n_1+n_2) \times (n_1+n_2)}$ partitioned as
\begin{align*}
    \mtx{A} = \blockmat{\mtx{A}_{11}}{\mtx{A}_{12}}{\mtx{A}_{21}}{\mtx{A}_{22}},\quad\text{with}\quad \mtx{A}_{11} \in \R^{n_1 \times n_1},\quad\mtx{A}_{12}\in\R^{n_1\times n_2},\quad \mtx{A}_{21} \in \R^{n_2 \times n_1},\quad\text{and}\quad\mtx{A}_{22} \in \R^{n_2 \times n_2} \:,
\end{align*}
we define the linear operators $\Projdiag$ and $\Projoff$ from $\R^{(n_1+n_2) \times
(n_1+n_2)} \longrightarrow \R^{(n_1+n_2) \times (n_1+n_2)}$ as follows
\begin{align*}
    \Projdiag(\mtx{A}) := \blockmat{\mtx{A}_{11}}{\mtx{0}_{n_1 \times n_2}}{\mtx{0}_{n_2 \times n_1}}{\mtx{A}_{22}}, \quad
    \Projoff(\mtx{A}) := \blockmat{\mtx{0}_{n_1 \times n_1}}{\mtx{A}_{12}}{\mtx{A}_{21}}{\mtx{0}_{n_2 \times n_2}} \:.
\end{align*}
Our final piece of notation is an augmented measurement map which works over
lifted matrices, which we call $\B$.  The map $\B :
\R^{(n_1+n_2)\times(n_1+n_2)} \longrightarrow \R^m$ is defined as
\begin{align*}
    \B(\mtx{X})_k := \ip{\mtx{B}_k}{\mtx{X}}, \qquad \mtx{B}_k := \Sym(\mtx{A}_k) \:.
\end{align*}
In this lifted space the function $g$ takes the form
\begin{align*}
  g(\mtx{W}):=g(\mtx{U},\mtx{V})&=\frac{1}{2} \twonorm{ \A(\mtx{U}\mtx{V}^\T) - \vct{b} }^2 + \frac{1}{16}\fronorm{\mtx{U}^\T\mtx{U} - \mtx{V}^\T\mtx{V}}^2\:\\
                               &=\frac{1}{4}\twonorm{\mathcal{B}\left(\text{Sym}\left(\mtx{U}\mtx{V}^T\right)-\text{Sym}\left(\mtx{M}\right)\right)}^2+\frac{1}{32}\fronorm{\text{Sym}\left(\mtx{U}\mtx{V}^T\right)-\text{Sym}\left(\mtx{M}\right)}^2
\end{align*}
Note that the updates of the Procrustes Flow algorithm in Equations \eqref{rectgraddescentupdate_a} and \eqref{rectgraddescentupdate_b} are based on the gradients $\nabla_{\mtx{U}} g(\mtx{U}, \mtx{V})$ and $\nabla_{\mtx{V}}
g(\mtx{U}, \mtx{V})$ given by
\begin{align*}
    \nabla_{\mtx{U}} g(\mtx{U}, \mtx{V}) &= \sum_{k=1}^{m} \ip{\mtx{A}_k}{\UVTMinusM} \mtx{A}_k \mtx{V} + \frac{1}{4} \mtx{U}(\mtx{U}^\T \mtx{U} - \mtx{V}^\T \mtx{V}) \\
    \nabla_{\mtx{V}} g(\mtx{U}, \mtx{V}) &= \sum_{k=1}^{m} \ip{\mtx{A}_k}{\UVTMinusM} \mtx{A}_k^\T \mtx{U} + \frac{1}{4} \mtx{V}(\mtx{V}^\T \mtx{V} - \mtx{U}^\T \mtx{U}) \:.
\end{align*}
One can easily verify that this update has the following compact
representation in terms of the lifted space
\begin{align*}
    \nabla g(\mtx{W}) = \blockvec{\nabla_{\mtx{U}} g(\mtx{U}, \mtx{V})}{\nabla_{\mtx{V}} g(\mtx{U}, \mtx{V})} = \frac{1}{2} \B^*\B( \mtx{WW}^\T - \Sym(\mtx{M})) \mtx{W} + \frac{1}{4} (\Projdiag - \Projoff)(\mtx{WW}^\T) \mtx{W} \:.
\end{align*}
As in the proof for the PSD case, the crux of Theorem~\ref{mainthm_general}
lies in establishing that the regularity condition
\begin{align}
\label{mainineqg}
\ip{\nabla g(\mtx{W})}{\mtx{W}-\mtx{Z}\mtx{R}} \geq \frac{\sigma_r(\mtx{M})}{8} \fronorm{\mtx{W}-\mtx{Z}\mtx{R}}^2 + \frac{16}{1683\opnorm{\mtx{M}}} \fronorm{\nabla g(\mtx{W})}^2 \:,
\end{align}
holds for all $\mtx{W}\in\R^{(n_1+n_2)\times r}$ obeying
$\text{dist}\left(\mtx{W}, \mtx{Z} \right)\le
\frac{1}{2\sqrt{2}}\sigma_r^{1/2}(\mtx{M})$.
Assuming that this condition holds, we have that $g(\mtx{W})$ obeys
$\RC(8/\sigma_r(\mtx{M}), \frac{1683}{16} \opnorm{\mtx{M}}, \frac{1}{2\sqrt{2}}
\sigma_r^{1/2}(\mtx{M}))$, and hence Theorem~\ref{mainthm_general},
Equation~\eqref{convthm_general} immediately follows by appealing to
Lemma~\ref{lemma:rc_implies_convergence}.

To prove \eqref{mainineqg}, we make use of the similarity of the expressions
with the PSD case. We start, as before, by defining a reference function
$F(\mtx{W}) := \frac{1}{4} \fronorm{ \WWTMinusZZT }^2$ with gradient $\nabla
F(\mtx{W}) = (\WWTMinusZZT) \mtx{W}$.
We now \COND{state two lemmas}{restate Lemmas~\ref{one514_summary} and \ref{two514_summary} from Section~\ref{sec:proofs:rect}}
relating $g$ and $F$, which together immediately imply
\eqref{mainineqg}.
The first lemma relates the regularity condition of $g$ to that of $F$ by
utilizing RIP. The second lemma provides a Lipschitz type property for the
gradient of $g$.
\begin{lemma}
\label{one514} Assume the linear mapping $\mathcal{A}$ obeys $4r$-RIP with constant $\delta_{4r}$. Then $g$ obeys the following regularity condition for any $\mtx{W} \in \R^{(n_1+n_2)\times r}$ and $\mtx{R} \in \R^{r \times r}$,
\begin{align}
  \ip{\nabla g(\mtx{W})}{\WMinusZR} &\geq -\frac{\delta_{4r}}{2} \fronorm{ \WWTMinusZZT } \fronorm{(\WMinusZR)\mtx{W}^\T} \nonumber \\
  &\qquad+ \frac{1}{4} \ip{\nabla F(\mtx{W})}{\WMinusZR} + \frac{1}{8\norm{\mtx{M}} } \fronorm{\widetilde{\mtx{Z}}\widetilde{\mtx{Z}}^\T\mtx{W} }^2 \:.
  \label{eq:rect:inequalityone}
\end{align}
\end{lemma}
\begin{lemma}
\label{two514}
Let $\mathcal{A}$ be a linear map obeying rank-$6r$ RIP with constant
$\delta_{6r} \leq 1/10$.  Then for all $\mtx{W} \in \R^{(n_1+n_2) \times r}$
satisfying $\dist(\mtx{W}, \mtx{Z}) \leq \frac{1}{4} \opnorm{\mtx{Z}}$, we have
that
\begin{align}
\label{ineqlem2}
\frac{21}{400}\fronorm{\WWTMinusZZT}^2 + \frac{1}{8\norm{\mtx{M}} } \fronorm{\widetilde{\mtx{Z}}\widetilde{\mtx{Z}}^\T\mtx{W} }^2 \ge \frac{16}{1683}\frac{1}{\opnorm{\mtx{M}}}\fronorm{\nabla g(\mtx{W})}^2.
\end{align}
\end{lemma}
With these lemmas in place we have all the elements to prove \eqref{mainineqg}.
We use Lemma~\ref{one514}, Equation~\eqref{eq:rect:inequalityone} together with
the inequality $2ab \leq a^2 + b^2$ to conclude that,
\begin{align}
\label{inter1}
  \ip{\nabla g(\mtx{W})}{&\WMinusZR} \geq - \frac{\delta_{4r}}{4}\left( \fronorm{\WWTMinusZZT}^2 + \fronorm{(\WMinusZR)\mtx{W}^\T}^2 \right) \nonumber \\
                                    &\qquad+ \frac{1}{4} \ip{\nabla F(\mtx{W})}{\WMinusZR} + \frac{1}{8\norm{\mtx{M}} } \fronorm{\widetilde{\mtx{Z}}\widetilde{\mtx{Z}}^\T\mtx{W} }^2.
\end{align}
By assumption $\dist(\mtx{W}, \mtx{Z}) \leq \frac{1}{4} \sigma_r(\mtx{Z})$, so
we can apply Lemma~\ref{lemma:expected_f_rc_cond} to $\ip{\nabla
F(\mtx{W})}{\WMinusZR}$, which combined with \eqref{inter1} yields
\begin{align}
\label{inter2}
  \ip{\nabla g(\mtx{W})}{&\WMinusZR} \geq (\frac{1}{100}-\frac{\delta_{4r}}{4}) \fronorm{\WWTMinusZZT}^2+(\frac{1}{80}-\frac{\delta_{4r}}{4}) \fronorm{(\WMinusZR)\mtx{W}^\T}^2 \nonumber \\
                                    &\qquad+ \frac{\sigma_r(\mtx{M})}{8} \fronorm{\WMinusZR}^2 +\frac{21}{400}\fronorm{\WWTMinusZZT}^2+ \frac{1}{8\norm{\mtx{M}} } \fronorm{\widetilde{\mtx{Z}}\widetilde{\mtx{Z}}^\T\mtx{W} }^2.
\end{align}
Applying Lemma \ref{two514} together with $\delta_{4r} \leq 1/25$ to
\eqref{inter2} completes the proof of \eqref{mainineqg} and hence the theorem.
All that remains is to prove Lemma \ref{one514} and \ref{two514}, which we do
in Sections \ref{secone514} and \ref{sectwo514}, respectively.

\subsubsection{Relating the regularity condition of $g$ and $F$ (Lemma \ref{one514})}
\label{secone514}
We begin the proof of Lemma \ref{one514} with the following RIP inequality
about the map $\B$.  The proof of this lemma is almost identical to the proof of
Lemma~\ref{lemma:rip:innerproducts}, so we omit the details.
\begin{lemma}
\label{lemma:rip:rect:innerproducts}
Suppose $\A$ is $2r$-RIP with constant $\delta_{2r}$, and $\B$ is constructed
from $\A$ as described above.  For any rank-$r$ matrices $\mtx{X}, \mtx{Y} \in
\R^{(n_1+n_2) \times (n_1+n_2)}$, we have
\begin{align*}
    \abs{\ip{\B(\mtx{X})}{\B(\mtx{Y})} - \ip{\Projoff(\mtx{X})}{\Projoff(\mtx{Y})}} \leq \delta_{2r} \fronorm{\Projoff(\mtx{X})} \fronorm{\Projoff(\mtx{Y})} \:.
\end{align*}
\end{lemma}
To relate the gradients $\nabla g(\mtx{W})$ and $\nabla F(\mtx{W})$ we first make a few manipulations to $\nabla g(\mtx{W})$. Define $\mtx{\Delta} :=  \mtx{WW}^\T - \Sym(\mtx{M})$. We have
\begin{align}
  \nabla g(\mtx{W}) &= \frac{1}{2} \B^*\B(\mtx{\Delta}) \mtx{W} + \frac{1}{4} (\Projdiag - \Projoff)(\mtx{WW}^\T) \mtx{W} \nonumber \\
                    &= \frac{1}{2} (\B^*\B(\mtx{\Delta}) - \Projoff(\mtx{\Delta})) \mtx{W} + \frac{1}{2} \Projoff(\mtx{\Delta}) \mtx{W} + \frac{1}{4} (\Projdiag - \Projoff)(\mtx{WW}^\T) \mtx{W} \nonumber \\
                    &= \frac{1}{2} (\B^*\B(\mtx{\Delta}) - \Projoff(\mtx{\Delta})) \mtx{W} + \frac{1}{2}\Projoff(\mtx{WW}^\T) \mtx{W} - \frac{1}{2}\Projoff(\Sym(\mtx{M})) \mtx{W} \nonumber \\
                    &\qquad\qquad+ \frac{1}{4}\Projdiag(\mtx{WW}^\T) \mtx{W} - \frac{1}{4} \Projoff(\mtx{WW}^\T) \mtx{W} \nonumber \\
                    &= \frac{1}{2} (\B^*\B(\mtx{\Delta}) - \Projoff(\mtx{\Delta}))\mtx{W} + \frac{1}{4}(\Projdiag + \Projoff)(\mtx{WW}^\T) \mtx{W} - \frac{1}{2}\Sym(\mtx{M}) \mtx{W} \nonumber\\
                    &= \frac{1}{2} (\B^*\B(\mtx{\Delta}) - \Projoff(\mtx{\Delta}))\mtx{W} + \frac{1}{4}( \mtx{WW}^\T - 2\Sym(\mtx{M}) ) \mtx{W}. \label{eq:rect:gradone}
\end{align}
Taking inner products of both sides of Equation~\eqref{eq:rect:gradone} gives us
\begin{align}
\label{myinter510}
  \ip{\nabla g(\mtx{W})}{\WMinusZR} = \frac{1}{2} \ip{(\B^*\B(\mtx{\Delta}) - \Projoff(\mtx{\Delta}))\mtx{W}}{\WMinusZR} + \frac{1}{4} \ip{(\mtx{WW}^\T - 2\Sym(\mtx{M}))\mtx{W}}{\WMinusZR} \:.
\end{align}
The first term is simple to control with RIP. Observe that
\begin{align}
  \ip{(\B^*\B(\mtx{\Delta}) &- \Projoff(\mtx{\Delta}))\mtx{W}}{\WMinusZR} \nonumber \\
                            &= \ip{\B^*\B(\mtx{\Delta}) - \Projoff(\mtx{\Delta})}{(\WMinusZR)\mtx{W}^\T} \nonumber \\
                            &= \ip{\B(\mtx{\Delta})}{\B((\WMinusZR)\mtx{W}^\T)} - \ip{\Projoff(\mtx{\Delta})}{\Projoff((\WMinusZR)\mtx{W}^\T)} \nonumber \\
                            &\stackrel{(a)}{\geq} -\delta_{4r} \fronorm{ \WWTMinusZZT } \fronorm{(\WMinusZR)\mtx{W}^\T} \:, \label{eq:rect:rip}
\end{align}
where (a) follows from Lemma~\ref{lemma:rip:rect:innerproducts}.

We now relate the second term to the gradient of $F$.  By exploiting
the structure of $\mtx{Z}$ and $\widetilde{\mtx{Z}}$, we have
\begin{align}
  \ip{\mtx{WW}^\T &- 2\Sym(\mtx{M}) }{\WMinusZR}  \nonumber \\
                  &\stackrel{(a)}{=} \ip{(\WWTMinusZZT)\mtx{W}}{\WMinusZR} + \ip{\widetilde{\mtx{Z}}  \widetilde{\mtx{Z}}^\T \mtx{W}}{\WMinusZR}  \nonumber \\
                  &\stackrel{(b)}{=} \ip{(\WWTMinusZZT)\mtx{W}}{\WMinusZR} + \Tr( \mtx{W}^\T \widetilde{\mtx{Z}}\widetilde{\mtx{Z}}^\T\mtx{W} )  \nonumber \\
                  &\stackrel{(c)}{\geq} \ip{(\WWTMinusZZT)\mtx{W}}{\WMinusZR} + \frac{1}{\norm{\widetilde{\mtx{Z}}}^2} \fronorm{\widetilde{\mtx{Z}}\widetilde{\mtx{Z}}^\T\mtx{W} }^2  \nonumber \\
                  &\stackrel{(d)}{=} \ip{\nabla F(\mtx{W})}{\WMinusZR} + \frac{1}{2\opnorm{\mtx{M}}} \fronorm{\widetilde{\mtx{Z}}\widetilde{\mtx{Z}}^\T\mtx{W} }^2 \:, \label{myinter510_b}
\end{align}
where (a) holds because $2\Sym(\mtx{M}) = \mtx{ZZ}^\T -
\widetilde{\mtx{Z}}\widetilde{\mtx{Z}}^\T$, (b) holds because
$\widetilde{\mtx{Z}}^\T \mtx{Z} = \mtx{0}_{r \times r}$, (c) holds because
$\fronorm{\widetilde{\mtx{Z}}\widetilde{\mtx{Z}}^\T\mtx{W} }^2 = \Tr(\mtx{W}^\T
\widetilde{\mtx{Z}}\widetilde{\mtx{Z}}^\T\widetilde{\mtx{Z}}\widetilde{\mtx{Z}}^\T
\mtx{W}) \leq \sigma_1(\widetilde{\mtx{Z}}^\T \widetilde{\mtx{Z}} ) \Tr(
\mtx{W}^\T \widetilde{\mtx{Z}}\widetilde{\mtx{Z}}^\T\mtx{W} )$, and (d) holds
since $\norm{\widetilde{\mtx{Z}}}^2 = 2\norm{\mtx{M}}$.
The proof of Lemma~\ref{one514} now follows from combining \eqref{myinter510}
with \eqref{eq:rect:rip} and \eqref{myinter510_b} .

\subsubsection{Lipschitz-gradient type condition for $g$ (Lemma \ref{two514})}
\label{sectwo514}
The left-hand side of \eqref{ineqlem2} has two terms. We start by bounding the
second term. Fix any $\varepsilon > 0$. Then,
\begin{align}
  \frac{1}{8\opnorm{\mtx{M}}} \fronorm{\widetilde{\mtx{Z}}\widetilde{\mtx{Z}}^\T\mtx{W} }^2 &= \frac{1}{8\opnorm{\mtx{M}}} \fronorm{  (\Projdiag-\Projoff)(\mtx{WW}^\T) \mtx{W} + (\Projdiag-\Projoff)(\mtx{ZZ}^\T - \mtx{WW}^\T)\mtx{W} }^2 \nonumber \\
                                                                                            &\stackrel{(a)}{\geq} \frac{1}{8\opnorm{\mtx{M}}} \left( \frac{\varepsilon}{1+\varepsilon} \fronorm{(\Projdiag-\Projoff)(\mtx{WW}^\T) \mtx{W} }^2 -  \varepsilon \fronorm{(\Projdiag-\Projoff)(\mtx{ZZ}^\T - \mtx{WW}^\T)\mtx{W}}^2\right) \nonumber \\
                                                                                            &\geq \frac{1}{8\opnorm{\mtx{M}}}\frac{\varepsilon}{1+\varepsilon}\fronorm{(\Projdiag-\Projoff)(\mtx{WW}^\T) \mtx{W} }^2 - \frac{\varepsilon}{8} \frac{\opnorm{\mtx{W}}^2}{\opnorm{\mtx{M}}} \fronorm{\WWTMinusZZT}^2 \nonumber \\
                                                                                            &\stackrel{(b)}{\geq} \frac{1}{8\opnorm{\mtx{M}}}\frac{\varepsilon}{1+\varepsilon}\fronorm{(\Projdiag-\Projoff)(\mtx{WW}^\T) \mtx{W} }^2 - \varepsilon \frac{25}{64} \fronorm{\WWTMinusZZT}^2 \:. \label{eq:rect:regtwo}
\end{align}
Here, (a) holds since by Young's inequality for any $\varepsilon>0$, we have $(a-b)^2 \geq
\frac{\varepsilon}{1+\varepsilon}a^2 - \varepsilon b^2$ and (b) holds since $2\opnorm{\mtx{M}}=\opnorm{\mtx{Z}}^2$ and $\opnorm{\mtx{W}} \leq \frac{5}{4} \opnorm{\mtx{Z}}$.

To bound the first term in left-hand side of \eqref{ineqlem2}, we state a lemma which shows that our augmented measurement map $\B$
obeys a similar Lipschitz property to that of $\A$ stated in Lemma~\ref{lemma:concentration:gradientnorm}. The proof of this lemma is nearly identical to that of
Lemma~\ref{lemma:concentration:gradientnorm}, and requires minor modifications
to deal with the projection operator $\Projoff$.  We omit the
details.
\begin{lemma}
\label{lemma:rect:concentration:gradientnorm}
Let $\A$ be as in the hypothesis of Lemma~\ref{lemma:concentration:gradientnorm}.
Then for all $\mtx{W}, \mtx{Z} \in \R^{(n_1+n_2)\times r}$, we have
\begin{align*}
  \fronorm{\Projoff(\WWTMinusZZT)}^2 \geq \frac{10}{17} \frac{1}{\opnorm{\mtx{W}}^2} \fronorm{ \B^*\B(\WWTMinusZZT) \mtx{W} }^2 \:.
\end{align*}
\end{lemma}
With this lemma in place, note that for any $\gamma > 0$,
\begin{align}
  \frac{1}{20} \fronorm{\WWTMinusZZT}^2 &\stackrel{(a)}{\geq} \frac{1}{34 \opnorm{\mtx{W}}^2} \fronorm{ \B^*\B(\WWTMinusZZT) \mtx{W} }^2  \nonumber \\
                                        &\stackrel{(b)}{\geq} \frac{4}{425 \opnorm{\mtx{M}}} \fronorm{ \B^*\B(\WWTMinusZZT) \mtx{W} }^2  \nonumber \\
                                        &= \frac{16}{425 \opnorm{\mtx{M}}} \fronorm{ \frac{1}{2} \B^*\B(\mtx{\Delta}) \mtx{W}   }^2  \nonumber \\
                                        &= \frac{16}{425 \opnorm{\mtx{M}}} \fronorm{ \nabla g(\mtx{W}) - \frac{1}{4} (\Projdiag - \Projoff)(\mtx{WW}^\T) \mtx{W}   }^2  \nonumber \\
                                        &\stackrel{(c)}{\geq} \frac{16}{425 \opnorm{\mtx{M}}} \left( \frac{\gamma}{1+\gamma} \fronorm{\nabla g(\mtx{W})}^2 - \frac{\gamma}{16} \fronorm{ (\Projdiag - \Projoff)(\mtx{WW}^\T) \mtx{W} }^2 \right)  \nonumber \\
                                        &= \frac{16}{425 \opnorm{\mtx{M}}} \frac{\gamma}{1+\gamma} \fronorm{ \nabla g(\mtx{W}) }^2 - \frac{\gamma}{425 \opnorm{\mtx{M}}} \fronorm{ (\Projdiag - \Projoff)(\mtx{WW}^\T) \mtx{W} }^2 \:. \label{eq:rect:regthree}
\end{align}
Here, (a) follows from Lemma~\ref{lemma:rect:concentration:gradientnorm},
(b) follows because $\opnorm{\mtx{W}} \leq \frac{5}{4} \opnorm{\mtx{Z}}$,
and (c) is another application of Young's inequality.
Combining \eqref{eq:rect:regtwo} and
\eqref{eq:rect:regthree} with the hypothesis that $\delta_{4r} \leq 1/25$, and
setting $\varepsilon = 4/625$, $\gamma = 25/74$ completes the proof.

\subsubsection{Proofs for the initialization phase of Algorithm \ref{alg:rpf} (Theorem \ref{mainthm_general}, Equation \eqref{initthm_general})}
\label{secinitgen}
We start with the following generalization of Lemma~\ref{lemma:dist:upperbound}, the
proof of which is deferred to Appendix~\ref{proofgeneralperturb}.
\begin{lemma}
\label{lemma:generalperturb}
Let $\mtx{M}_1, \mtx{M}_2 \in \R^{n_1 \times n_2}$ be two rank $r$ matrices
with SVDs of the form $\mtx{M}_1 = \mtx{U}_1\mtx{\Sigma}_1\mtx{V}_1^\T$ and
$\mtx{M}_2 = \mtx{U}_2 \mtx{\Sigma}_2 \mtx{V}_2^\T$.  For $\ell=1,2$, define $\mtx{X}_\ell = \mtx{U}_\ell\mtx{\Sigma}_\ell^{1/2} \in \R^{n_1 \times r}$
and $\mtx{Y}_\ell = \mtx{V}_\ell\mtx{\Sigma}_\ell^{1/2} \in \R^{n_2 \times r}$. Furthermore, assume $\mtx{M}_1$ and $\mtx{M}_2$ obey $\opnorm{\mtx{M}_2 - \mtx{M}_1} \leq \frac{1}{2} \sigma_r(\mtx{M}_1)$.
Under these assumptions the following inequality holds
\begin{align*}
  \dist^2\left(\blockvec{\mtx{X}_2}{\mtx{Y}_2}, \blockvec{\mtx{X}_1}{\mtx{Y}_1}\right) \leq \frac{2}{\sqrt{2}-1} \frac{\fronorm{\mtx{M}_2 - \mtx{M}_1}^2}{\sigma_r(\mtx{M}_1)} \:.
\end{align*}
\end{lemma}
The rest of the proof proceeds similarly to the proof of
Equation~\eqref{initthm}.
Using Lemma~\ref{lemma:rip:innerproducts}, we conclude that
$\rho(\A)$ from Lemma~\ref{lemma:hardthresholding} is bounded by
$\rho(\A) \leq 2 \delta_{4r} \leq 2/25$. Setting
$\widetilde{\mtx{M}}_0 = \mtx{0}_{n_1 \times n_2}$ and applying Lemma~\ref{lemma:hardthresholding}
to our initialization iterates, we have that
\begin{align}
  \fronorm{ \widetilde{\mtx{M}}_\tau - \mtx{M} } \leq (2/25)^{\tau} \fronorm{\mtx{M}} \:. \label{eq:rect:initone}
\end{align}
In order for the RHS to be bounded above by $\frac{1}{2} \sigma_r(\mtx{M})$,
$\tau$ must satisfy
\begin{align}
  \tau \geq \log(25/2) \log\left(\frac{1}{2} \cdot \frac{\fronorm{\mtx{M}}}{\sigma_r(\mtx{M})}  \right) \:. \label{eq:rect:inittwo}
\end{align}
When this happens, Lemma~\ref{lemma:generalperturb} tells us that
\begin{align*}
  \dist^2(\mtx{W}_0, \mtx{Z}) \leq \frac{2}{\sqrt{2}-1} \frac{\fronorm{ \widetilde{\mtx{M}}_\tau - \mtx{M} }^2}{\sigma_r(\mtx{M})} \:.
\end{align*}
In order for this RHS to be bounded above by $\frac{1}{8} \sigma_r(\mtx{M})$,
we require that
$\fronorm{ \widetilde{\mtx{M}}_\tau - \mtx{M} }^2 \leq \frac{\sqrt{2}-1}{16} \sigma_r^2(\mtx{M})$.
Using Equation~\eqref{eq:rect:initone}, it is sufficient for $\tau$ to satisfy
\begin{align}
  \tau \geq \log(25/2) \log\left( 7 \cdot \frac{\fronorm{\mtx{M}}}{\sigma_r(\mtx{M})}  \right) \:. \label{eq:rect:initthree}
\end{align}
Since $\fronorm{\mtx{M}} \leq \sqrt{r} \opnorm{\mtx{M}}$, setting $\tinit$ as
$\tinit \geq 3 \log(\sqrt{r} \kappa) + 5$
satisfies both \eqref{eq:rect:inittwo} and \eqref{eq:rect:initthree}.

\section*{Acknowledgements}
BR is generously supported by ONR awards N00014-11-1-0723 and N00014-13-1-0129,
NSF awards CCF-1148243 and CCF-1217058, AFOSR award FA9550-13-1-0138, and a
Sloan Research Fellowship.  RB is generously supported by ONR award N00014-11-1-0723 and the NDSEG Fellowship.
This research is supported in part by NSF CISE Expeditions Award
CCF-1139158, LBNL Award 7076018, and DARPA XData Award FA8750-12-2-0331, and
gifts from Amazon Web Services, Google, SAP, The Thomas and Stacey Siebel
Foundation, Adatao, Adobe, Apple, Inc., Blue Goji, Bosch, C3Energy, Cisco,
Cray, Cloudera, EMC2, Ericsson, Facebook, Guavus, HP, Huawei, Informatica,
Intel, Microsoft, NetApp, Pivotal, Samsung, Schlumberger, Splunk, Virdata and
VMware.

\else
\subsection{Rectangular case}

\label{sec:proofs:rect}

We now turn our attention to the general case where the matrices are
rectangular.  Recall that in this case, we want to recover a fixed but unknown
rank-$r$ matrix $\mtx{M} \in \R^{n_1 \times n_2}$ from linear measurements.
Assume that $\mtx{M}$ has a singular value decomposition of the form $\mtx{M} =
\mtx{A} \mtx{\Sigma} \mtx{B}^\T$.  Define $\mtx{X} = \mtx{A} \mtx{\Sigma}^{1/2}
\in \R^{n_1 \times r}$ and $\mtx{Y} = \mtx{B} \mtx{\Sigma}^{1/2} \in \R^{n_2
\times r}$.  With this piece of notation the iterates $\mtx{U}_\tau \in \R^{n_1
\times r}, \mtx{V}_\tau \in \R^{n_2 \times r}$ in Algorithm~\ref{alg:rpf} can
be thought of as estimates of $\mtx{X}$ and $\mtx{Y}$. The proof of the
correctness of the initialization phase of Procrustes Flow (Theorem
\ref{mainthm_general}, Equation \eqref{initthm_general}) in the rectangular case
is similar to the PSD case (Theorem \ref{mainthm}, Equation
\eqref{initthm}) and is detailed in Section \ref{secinitgen}. In this section
we shall describe the main ideas of the proof. See Section~\ref{secgeneral} for the
full details.

To simplify exposition we aggregate the pairs of matrices $(\mtx{U},\mtx{V})$, $(\mtx{X},\mtx{Y})$, and $(\mtx{X},-\mtx{Y})$ into larger ``lifted" matrices as follows
\begin{align*}
\mtx{W} := \blockvec{\mtx{U}}{\mtx{V}},\quad\mtx{Z} := \blockvec{\mtx{X}}{\mtx{Y}},\quad\text{and}\quad \widetilde{\mtx{Z}} := \blockvec{\mtx{X}}{-\mtx{Y}}.
\end{align*}
To prove Theorem~\ref{mainthm_general}, Equation~\eqref{convthm_general}, we
will demonstrate that the function $g(\mtx{W}):=g(\mtx{U},\mtx{V})$ over the
variable $\mtx{W}$ has similar form to $f(\mtx{U})$ over the variable
$\mtx{U}$. 
As in the proof for the PSD case, the crux of Theorem~\ref{mainthm_general}
lies in establishing that the regularity condition
\begin{align}
\label{mainineqg_summary}
\ip{\nabla &g(\mtx{W})}{\mtx{W}-\mtx{Z}\mtx{R}}  \nonumber \\
&\geq \frac{\sigma_r(\mtx{M})}{8} \fronorm{\mtx{W}-\mtx{Z}\mtx{R}}^2 + \frac{16}{1683\opnorm{\mtx{M}}} \fronorm{\nabla g(\mtx{W})}^2 \:,
\end{align}
holds for all $\mtx{W}\in\R^{(n_1+n_2)\times r}$ obeying
$\text{dist}\left(\mtx{W}, \mtx{Z} \right)\le
\frac{1}{2\sqrt{2}}\sigma_r^{1/2}(\mtx{M})$.
Assuming that this condition holds, we have that $g(\mtx{W})$ obeys
$\RC(8/\sigma_r(\mtx{M}), \frac{1683}{16} \opnorm{\mtx{M}}, \frac{1}{2\sqrt{2}}
\sigma_r^{1/2}(\mtx{M}))$, and hence Theorem~\ref{mainthm_general},
Equation~\eqref{convthm_general} immediately follows by appealing to
Lemma~\ref{lemma:rc_implies_convergence}.

To prove \eqref{mainineqg_summary}, we make use of the similarity of the expressions
with the PSD case. We start, as before, by defining a reference function
$F(\mtx{W}) := \frac{1}{4} \fronorm{ \WWTMinusZZT }^2$ with gradient $\nabla
F(\mtx{W}) = (\WWTMinusZZT) \mtx{W}$.
We now state two lemmas relating $g$ and $F$, which together immediately imply
\eqref{mainineqg_summary}.
The first lemma relates the regularity condition of $g$ to that of $F$ by
utilizing RIP. The second lemma provides a Lipschitz type property for the
gradient of $g$.
\COND{}{The proofs can be found in 
Sections \ref{secone514} and \ref{sectwo514}.}
\begin{lemma}
\label{one514_summary} Assume the linear mapping $\mathcal{A}$ obeys $4r$-RIP with constant $\delta_{4r}$. Then $g$ obeys the following regularity condition for any $\mtx{W} \in \R^{(n_1+n_2)\times r}$ and $\mtx{R} \in \R^{r \times r}$,
\begin{align}
  \ip{\nabla &g(\mtx{W})}{\WMinusZR} \nonumber \\
  &\geq -\frac{\delta_{4r}}{2} \fronorm{ \WWTMinusZZT } \fronorm{(\WMinusZR)\mtx{W}^\T} \nonumber \\
  &\qquad+ \frac{1}{4} \ip{\nabla F(\mtx{W})}{\WMinusZR} + \frac{1}{8\norm{\mtx{M}} } \fronorm{\widetilde{\mtx{Z}}\widetilde{\mtx{Z}}^\T\mtx{W} }^2 \:.
  \label{eq:rect:inequalityone_summary}
\end{align}
\end{lemma}
\begin{lemma}
\label{two514_summary}
Let $\mathcal{A}$ be a linear map obeying rank-$6r$ RIP with constant
$\delta_{6r} \leq 1/10$.  Then for all $\mtx{W} \in \R^{(n_1+n_2) \times r}$
satisfying $\dist(\mtx{W}, \mtx{Z}) \leq \frac{1}{4} \opnorm{\mtx{Z}}$, we have
that
\begin{align}
\label{ineqlem2_summary}
\frac{21}{400}\fronorm{\WWTMinusZZT}^2 &+ \frac{1}{8\norm{\mtx{M}} } \fronorm{\widetilde{\mtx{Z}}\widetilde{\mtx{Z}}^\T\mtx{W} }^2 \nonumber \\
&\ge \frac{16}{1683}\frac{1}{\opnorm{\mtx{M}}}\fronorm{\nabla g(\mtx{W})}^2.
\end{align}
\end{lemma}
With these lemmas in place we have all the elements to prove
\eqref{mainineqg_summary}.  By applying Lemma~\ref{lemma:expected_f_rc_cond} to
$\ip{\nabla F(\mtx{W})}{\WMinusZR}$ and combining
\eqref{eq:rect:inequalityone_summary} and \eqref{ineqlem2_summary},
Equation~\eqref{mainineqg_summary} follows after some simple manipulations.
This concludes the proof of Theorem~\ref{mainthm_general}.

\fi

\if\MODE2
\bibliography{p}
\bibliographystyle{icml2016}

\clearpage
\appendix
\onecolumn

\icmltitle{Supplementary Material}

\section{Proof of Lemma~\ref{lemma:dist:upperbound}}
\label{proofstdlem}
Define $\mtx{H}=\mtx{U}-\mtx{X}\mtx{R}$. Similar to the discussion at the beginning of Section~\ref{sec:proofs:first:regularity}, without loss of generality we can assume that (a) $\mtx{R}=\mtx{I}$, (b) $\mtx{U}^\T \mtx{X} \succeq 0$, and (c) $\mtx{H}^\T \mtx{X} = \mtx{X}^\T \mtx{H}$. With these simplifications, establishing the lemma is equivalent to showing that
\begin{align}
    \Tr( (\mtx{H}^\T\mtx{H})^2 + 4 \mtx{H}^\T\mtx{H}\mtx{H}^\T\mtx{X} + 2 (\mtx{H}^\T\mtx{X})^2 + 2 \mtx{X}^\T\mtx{X}\mtx{H}^\T\mtx{H} - \eta \mtx{H}^\T\mtx{H} ) \geq 0 \label{eq:lifted_equiv_cond}
\end{align}
holds with $\eta=\frac{1}{2(\sqrt{2}-1) \sigma_r^2(\mtx{X})}$. We note that
\begin{align*}
    \Tr( (\mtx{H}^\T\mtx{H} &+ \sqrt{2}\mtx{H}^\T\mtx{X})^2 + (4 - 2\sqrt{2}) \mtx{H}^\T\mtx{H}\mtx{H}^\T\mtx{X} + 2  \mtx{X}^\T\mtx{X}\mtx{H}^\T\mtx{H} - \eta \mtx{H}^\T\mtx{H}) \\
                     &= \Tr( (\mtx{H}^\T\mtx{H})^2 + 4 \mtx{H}^\T\mtx{H}\mtx{H}^\T\mtx{X} + 2 (\mtx{H}^\T\mtx{X})^2 + 2  \mtx{X}^\T\mtx{X}\mtx{H}^\T\mtx{H} - \eta \mtx{H}^\T\mtx{H} ) \:.
\end{align*}
Hence, a sufficient condition for \eqref{eq:lifted_equiv_cond} to hold is
\begin{align}
    (4-2\sqrt{2})\mtx{H}^\T\mtx{X} + 2\mtx{X}^\T\mtx{X} - \eta \mtx{I}_r \succeq 0 \:. \label{eq:lifed_equiv_cond_two}
\end{align}
Recalling that $\mtx{H}^\T\mtx{X} = \mtx{U}^\T\mtx{X} -  \mtx{X}^\T\mtx{X}$, and that $\mtx{U}^\T\mtx{X} \succeq 0$, we have
\begin{align*}
    (4-2\sqrt{2})\mtx{H}^\T\mtx{X} + 2  \mtx{X}^\T\mtx{X} - \eta \mtx{I}_r &= (4-2\sqrt{2})\mtx{U}^\T\mtx{X} + (2 - (4-2\sqrt{2}))  \mtx{X}^\T\mtx{X} - \eta \mtx{I}_r \\
                                                  &= (4-2\sqrt{2})\mtx{U}^\T\mtx{X} + 2(\sqrt{2}-1)  \mtx{X}^\T\mtx{X} - \eta \mtx{I}_r \:.
\end{align*}
Since $\mtx{U}^\T\mtx{X} \succeq 0$, to show \eqref{eq:lifed_equiv_cond_two} it suffices to show
\begin{align*}
    2(\sqrt{2}-1)  \mtx{X}^\T\mtx{X} - \eta \mtx{I}_r \succeq 0 \Longleftrightarrow \mtx{X}^\T\mtx{X} \succeq \frac{\eta}{2(\sqrt{2}-1)} \mtx{I}_r \:.
\end{align*}
The RHS trivially holds, concluding the proof.

\section{Proof of Lemma~\ref{lemma:stopinit}}
\label{proofstopinit}
From RIP and the assumption that $\delta_{2r}\leq 1/10$, we have
\begin{align*}
  \opnorm{ \widetilde{\mtx{M}}_\tau - \XXT } \leq \fronorm{ \widetilde{\mtx{M}}_\tau - \XXT } \leq \sqrt{\frac{10}{9}} e_\tau \:.
\end{align*}
By Weyl's inequalities, this means that
\begin{align}
\label{eq:stone}
    \sigma_r^2(\mtx{X}) \geq \sigma_r(\widetilde{\mtx{M}}_\tau) -  \sqrt{\frac{10}{9}}e_\tau \:.
\end{align}
Lemma~\ref{lemma:dist:upperbound} ensures that
\begin{align*}
  \dist(\mtx{U}_0, \mtx{X}) \leq \sqrt{\frac{3}{2}} \frac{1}{\sigma_r(\mtx{X})} \fronorm{ \widetilde{\mtx{M}}_\tau - \XXT } \:.
\end{align*}
We can upper bound the RHS by the following chain of inequalities,
\begin{align*}
  \sqrt{\frac{3}{2}} \frac{1}{\sigma_r(\mtx{X})} \fronorm{ \widetilde{\mtx{M}}_\tau - \XXT } &\stackrel{(a)}{\leq} \sqrt{\frac{3}{2}} \frac{1}{\sigma_r(\mtx{X})} \sqrt{\frac{10}{9}} e_\tau \\
                                                                                 &\stackrel{(b)}{\leq} \sqrt{\frac{3}{2}} \frac{1}{\sigma_r(\mtx{X})} \frac{1}{2\sqrt{6}} \left( \sigma_r(\widetilde{\mtx{M}}) - \sqrt{\frac{10}{9}} e_\tau \right) \\
                                                                                 &\stackrel{(c)}{\leq} \sqrt{\frac{3}{2}} \frac{1}{\sigma_r(\mtx{X})} \frac{1}{2\sqrt{6}} \sigma_r^2(\mtx{X})
                                                                                 = \frac{1}{4} \sigma_r(\mtx{X})
\end{align*}
where (a) follows from RIP,
(b) follows since $e_\tau \leq \frac{3}{20} \sigma_r(\widetilde{\mtx{M}})$ implies that
\begin{align*}
  \sqrt{\frac{10}{9}} e_\tau \leq \frac{1}{2\sqrt{6}} \left( \sigma_r(\widetilde{\mtx{M}}) -  \sqrt{\frac{10}{9}} e_\tau \right) \:,
\end{align*}
and (c) follows by \eqref{eq:stone}.

\section{Proof of Lemma~\ref{lemma:generalperturb}}
\label{proofgeneralperturb}

To begin with note that by the dilation trick we have for $\ell = 1,2$,
\begin{align*}
	\blockmatoff{\mtx{M}_\ell}{\mtx{M}_\ell^\T} = \frac{1}{2} \blockmat{\mtx{U}_\ell}{\mtx{U}_\ell}{\mtx{V}_\ell}{-\mtx{V}_\ell} \blockmatdiag{\mtx{\Sigma}_\ell}{-\mtx{\Sigma}_\ell}\blockmat{\mtx{U}_\ell}{\mtx{U}_\ell}{\mtx{V}_\ell}{-\mtx{V}_\ell}^\T \:.
\end{align*}
By simple algebraic manipulations we have
\begin{align}
\label{simple513}
	\blockmatoff{\mtx{M}_1}{\mtx{M}_1^\T} - \blockmatoff{\mtx{M}_2}{\mtx{M}_2^\T} = \frac{1}{2} \blockmat{\mtx{X}_1}{\mtx{X}_2}{\mtx{Y}_1}{-\mtx{Y}_2}\blockmat{\mtx{X}_1}{\mtx{X}_2}{\mtx{Y}_1}{-\mtx{Y}_2}^\T - \frac{1}{2} \blockmat{\mtx{X}_1}{\mtx{X}_2}{-\mtx{Y}_1}{\mtx{Y}_2}\blockmat{\mtx{X}_1}{\mtx{X}_2}{-\mtx{Y}_1}{\mtx{Y}_2}^\T \:.
\end{align}
Furthermore,
\begin{align}
\label{myeq513}
	\blockmat{\mtx{X}_1}{\mtx{X}_2}{-\mtx{Y}_1}{\mtx{Y}_2}\blockmat{\mtx{X}_1}{\mtx{X}_2}{-\mtx{Y}_1}{\mtx{Y}_2}^\T &= \blockmat{ \mtx{X}_1\mtx{X}_1^\T + \mtx{X}_2\mtx{X}_2^\T }{ -\mtx{X}_1\mtx{Y}_1^\T + \mtx{X}_2\mtx{Y}_2^\T  }{ -\mtx{Y}_1\mtx{X}_1^\T + \mtx{Y}_2\mtx{X}_2^\T  }{ \mtx{Y}_1\mtx{Y}_1^\T + \mtx{Y}_2\mtx{Y}_2^\T } \nonumber\\
															&= \blockmat{ \mtx{X}_1\mtx{X}_1^\T + \mtx{X}_2\mtx{X}_2^\T  }{ \mtx{0} }{\mtx{0}}{ \mtx{Y}_1\mtx{Y}_1^\T + \mtx{Y}_2\mtx{Y}_2^\T  } + \blockmatoff{\mtx{M}_2 - \mtx{M}_1}{\mtx{M}_2^\T - \mtx{M}_1^\T} \:.
\end{align}
Applying Weyl's inequality to \eqref{myeq513}, we have
\begin{align}
\label{sigmamin513}
	\sigma_{2r}\left(\blockmat{\mtx{X}_1}{\mtx{X}_2}{-\mtx{Y}_1}{\mtx{Y}_2}\blockmat{\mtx{X}_1}{\mtx{X}_2}{-\mtx{Y}_1}{\mtx{Y}_2}^\T \right) &\geq \sigma_{2r}\left(\blockmat{ \mtx{X}_1\mtx{X}_1^\T + \mtx{X}_2\mtx{X}_2^\T  }{ \mtx{0} }{\mtx{0}}{ \mtx{Y}_1\mtx{Y}_1^\T + \mtx{Y}_2\mtx{Y}_2^\T  }  \right) - \opnorm{\blockmatoff{\mtx{M}_2 - \mtx{M}_1}{\mtx{M}_2^\T - \mtx{M}_1^\T} } \nonumber\\
																		 &= \sigma_{2r}\left(\blockmat{ \mtx{X}_1\mtx{X}_1^\T + \mtx{X}_2\mtx{X}_2^\T  }{ \mtx{0} }{\mtx{0}}{ \mtx{Y}_1\mtx{Y}_1^\T + \mtx{Y}_2\mtx{Y}_2^\T} \right) - \opnorm{\mtx{M}_2 - \mtx{M}_1} \nonumber\\
												       &\geq \sigma_{2r}\left(\blockmat{ \mtx{X}_1\mtx{X}_1^\T  }{ \mtx{0} }{\mtx{0}}{ \mtx{Y}_1\mtx{Y}_1^\T} \right) - \opnorm{\mtx{M}_2 - \mtx{M}_1} \nonumber\\
												&\geq \frac{1}{2} \sigma_r(\mtx{M}_1) \:.
\end{align}
Applying Lemma~\ref{lemma:dist:upperbound} to the matrices
$\blockmat{\mtx{X}_1}{\mtx{X}_2}{\mtx{Y}_1}{-\mtx{Y}_2}$ and
$\blockmat{\mtx{X}_1}{\mtx{X}_2}{-\mtx{Y}_1}{\mtx{Y}_2}$ and utilizing equations \eqref{simple513} and \eqref{sigmamin513} we conclude that
\begin{align}
\label{inter513}
	\dist^2\left( \blockmat{\mtx{X}_1}{\mtx{X}_2}{\mtx{Y}_1}{-\mtx{Y}_2}, \blockmat{\mtx{X}_1}{\mtx{X}_2}{-\mtx{Y}_1}{\mtx{Y}_2}  \right) \leq \frac{4}{\sqrt{2}-1} \frac{\fronorm{\mtx{M}_2 - \mtx{M}_1}^2}{\sigma_r(\mtx{M}_1)}.
\end{align}
Let $\mtx{A}\mtx{S}\mtx{B}^\T$ be the singular value decomposition of $\mtx{X}_1^\T \mtx{X}_2 + \mtx{Y}_1^\T \mtx{Y}_2$.
It is easy to verify that the solution $\mtx{R}$ to the orthogonal Procrustes problem (equivalently the optimal rotation) between
$\blockmat{\mtx{X}_1}{\mtx{X}_2}{\mtx{Y}_1}{-\mtx{Y}_2}$ and
$\blockmat{\mtx{X}_1}{\mtx{X}_2}{-\mtx{Y}_1}{\mtx{Y}_2}$ is equal to
$\mtx{R} = \blockmatoff{\mtx{AB}^\T}{\mtx{BA}^\T}$. From this we conclude that
\begin{align*}
	\blockmat{\mtx{X}_1}{\mtx{X}_2}{\mtx{Y}_1}{-\mtx{Y}_2} - \blockmat{\mtx{X}_1}{\mtx{X}_2}{-\mtx{Y}_1}{\mtx{Y}_2} \mtx{R} = \blockmat{ \mtx{X}_1 - \mtx{X}_2 \mtx{BA}^\T}{ \mtx{X}_2 - \mtx{X}_1\mtx{AB}^\T }{ \mtx{Y}_1 - \mtx{Y}_2\mtx{BA}^\T }{ -\mtx{Y}_2 + \mtx{Y}_1\mtx{AB}^\T }.
\end{align*}
Therefore,
\begin{align*}
	\dist^2\left( \blockmat{\mtx{X}_1}{\mtx{X}_2}{\mtx{Y}_1}{-\mtx{Y}_2}, \blockmat{\mtx{X}_1}{\mtx{X}_2}{-\mtx{Y}_1}{\mtx{Y}_2}  \right)&=
	\fronorm{\blockmat{\mtx{X}_1}{\mtx{X}_2}{\mtx{Y}_1}{-\mtx{Y}_2} - \blockmat{\mtx{X}_1}{\mtx{X}_2}{-\mtx{Y}_1}{\mtx{Y}_2} \mtx{R} }^2\\
	 &= 2\fronorm{\mtx{X}_2 - \mtx{X}_1\mtx{AB}^\T}^2 + 2\fronorm{\mtx{Y}_2 - \mtx{Y}_1\mtx{AB}^\T}^2 \\
																	     &= 2 \cdot \dist^2\left( \blockvec{\mtx{X}_2}{\mtx{Y}_2}, \blockvec{\mtx{X}_1}{\mtx{Y}_1}  \right) \:.
\end{align*}
Plugging the latter in \eqref{inter513} concludes the proof.

\section{Proof of the first inequality in Equation~\eqref{myineq}}
\label{sec:appendix:simplecalc}

The first inequality is immediate from the following lemma.
\begin{lemma}
Let $\mtx{U}, \mtx{X} \in \R^{n_1 \times r}$ and $\mtx{V}, \mtx{Y} \in \R^{n_2 \times r}$.
Suppose that
\begin{align*}
    \dist\left( \blockvec{\mtx{U}}{\mtx{V}}, \blockvec{\mtx{X}}{\mtx{Y}}  \right) \leq \frac{1}{4} \opnorm{\blockvec{\mtx{X}}{\mtx{Y}}} \:.
\end{align*}
Then, we have that
\begin{align*}
    \fronorm{ \UVT - \XYT } \leq \frac{9}{4 \sqrt{2}} \opnorm{\blockvec{\mtx{X}}{\mtx{Y}}} \dist\left(  \blockvec{\mtx{U}}{\mtx{V}}, \blockvec{\mtx{X}}{\mtx{Y}}  \right) \:.
\end{align*}
\end{lemma}
\begin{proof}
Put $\mtx{W} := \blockvec{\mtx{U}}{\mtx{V}}$ and $\mtx{Z} := \blockvec{\mtx{X}}{\mtx{Y}}$.
We have that $\fronorm{ \UVT - \XYT } \leq \frac{1}{\sqrt{2}} \fronorm{\WWTMinusZZT}$.
Applying Lemma~\ref{simplem} to $\fronorm{\WWTMinusZZT}$, we conclude that
$\frac{1}{\sqrt{2}} \fronorm{\WWTMinusZZT} \leq \frac{9}{4 \sqrt{2}} \opnorm{\mtx{Z}} \dist(\mtx{W}, \mtx{Z})$.
The result now follows.
\end{proof}

\else
{\small
\bibliography{p}

\newcommand{\etalchar}[1]{$^{#1}$}
\begin{thebibliography}{LRS{\etalchar{+}}10}

\bibitem[AM07]{achlioptas2007fast}
D.~Achlioptas and F.~McSherry.
\newblock Fast computation of low-rank matrix approximations.
\newblock {\em Journal of the ACM}, 54(2), 2007.

\bibitem[BD09]{blumensath2009}
T.~Blumensath and M.~E. Davies.
\newblock Iterative hard thresholding for compressed sensing.
\newblock {\em Applied and Computational Harmonic Analysis}, 27(3):265--274,
  2009.

\bibitem[BKS15]{bhojanapalli15}
S.~Bhojanapalli, A.~Kyrillidis, and S.~Sanghavi.
\newblock Dropping convexity for faster semi-definite optimization.
\newblock {\em arXiv}, arXiv:1509.03917, 2015.

\bibitem[Can08]{candes08}
E.~J. Cand{\`{e}}s.
\newblock The restricted isometry property and its implications for compressed
  sensing.
\newblock {\em Compte Rendus de l'Academie des Sciences}, 2008.

\bibitem[CCS10]{cai2010singular}
J.~F. Cai, E.~J. Cand{\`{e}}s, and Z.~Shen.
\newblock A singular value thresholding algorithm for matrix completion.
\newblock {\em {SIAM} Journal on Optimization}, 20(4):1956--1982, 2010.

\bibitem[CLM15]{cai2015optimal}
T.~T. Cai, X.~Li, and Z.~Ma.
\newblock Optimal rates of convergence for noisy sparse phase retrieval via
  thresholded {W}irtinger flow.
\newblock {\em arXiv}, arXiv:1506.03382, 2015.

\bibitem[CLS14]{candes2014phase}
E.~J. Cand{\`{e}}s, X.~Li, and M.~Soltanolkotabi.
\newblock Phase retrieval from coded diffraction patterns.
\newblock {\em Applied and Computational Harmonic Analysis}, 39(2):277--299,
  2014.

\bibitem[CLS15]{candes14}
E.~J. Cand{\`{e}}s, X.~Li, and M.~Soltanolkotabi.
\newblock Phase retrieval via wirtinger flow: Theory and algorithms.
\newblock {\em {IEEE} Transactions on Information Theory}, 61(4):1985--2007,
  2015.

\bibitem[CP11]{CandesPlanTight}
E.~J. Cand\`es and Y.~Plan.
\newblock Tight oracle bounds for low-rank matrix recovery from a minimal
  number of random measurements.
\newblock {\em {IEEE} Transactions on Information Theory}, 57(4):2342--2359,
  2011.

\bibitem[CR09]{candes2009exact}
E.~J. Cand{\`{e}}s and B.~Recht.
\newblock Exact matrix completion via convex optimization.
\newblock {\em Foundations of Computational Mathematics}, 9(6):717--772, 2009.

\bibitem[CT05]{candes2005decoding}
E.~J. Cand{\`{e}}s and T.~Tao.
\newblock Decoding by linear programming.
\newblock {\em IEEE Transactions on Information Theory}, 51(12):4203--4215,
  2005.

\bibitem[CW15]{chen15}
Y.~Chen and M.~J. Wainwright.
\newblock Fast low-rank estimation by projected gradient descent: General
  statistical and algorithmic guarantees.
\newblock {\em arXiv}, arXiv:1509.03025, 2015.

\bibitem[DR16]{davenport16}
M.~A. Davenport and J.~Romberg.
\newblock An overview of low-rank matrix recovery from incomplete observations.
\newblock {\em arXiv}, arXiv:1601.06422, 2016.

\bibitem[Faz02]{fazel02}
M.~Fazel.
\newblock {\em Matrix Rank Minimization with Applications}.
\newblock PhD thesis, Stanford University, 2002.

\bibitem[Fun06]{SimonFunk}
S.~Funk.
\newblock Netflix update: Try this at home, December 2006.

\bibitem[GK09]{garg2009gradient}
R.~Garg and R.~Khandekar.
\newblock Gradient descent with sparsification: an iterative algorithm for
  sparse recovery with restricted isometry property.
\newblock In {\em ICML}, 2009.

\bibitem[Gro11]{gross2011recovering}
D.~Gross.
\newblock Recovering low-rank matrices from few coefficients in any basis.
\newblock {\em IEEE Transactions on Information Theory}, 57(3):1548--1566,
  2011.

\bibitem[Har14]{hardt14a}
Moritz Hardt.
\newblock Understanding alternating minimization for matrix completion.
\newblock In {\em FOCS}, 2014.

\bibitem[JNS13]{jain12}
P.~Jain, P.~Netrapalli, and S.~Sanghavi.
\newblock Low-rank matrix completion using alternating minimization.
\newblock In {\em STOC}, 2013.

\bibitem[Kes12]{keshavan2012efficient}
R.~H. Keshavan.
\newblock {\em Efficient algorithms for collaborative filtering}.
\newblock PhD thesis, Stanford University, 2012.

\bibitem[KMO10]{keshavan2010matrix}
R.~H. Keshavan, A.~Montanari, and S.~Oh.
\newblock Matrix completion from a few entries.
\newblock {\em IEEE Transactions on Information Theory}, 56(6):2980--2998,
  2010.

\bibitem[LRS{\etalchar{+}}10]{lee10b}
J.~Lee, B.~Recht, N.~Srebro, J.~A. Tropp, and R.~Salakhutdinov.
\newblock Practical large-scale optimization for max-norm regularization.
\newblock In {\em NIPS}, 2010.

\bibitem[MJD09]{meka09}
R.~Meka, P.~Jain, and I.~S. Dhillon.
\newblock Guaranteed rank minimization via singular value projection.
\newblock {\em arXiv}, arXiv:0909.5457, 2009.

\bibitem[NT09]{CoSamp}
D.~Needell and J.~A. Tropp.
\newblock {CoSaMP}: Iterative signal recovery from incomplete and inaccurate
  samples.
\newblock {\em Applied and Computational Harmonic Analysis}, 26(3):301--321,
  2009.

\bibitem[NV09]{needell2009uniform}
D.~Needell and R.~Vershynin.
\newblock Uniform uncertainty principle and signal recovery via regularized
  orthogonal matching pursuit.
\newblock {\em Foundations of Computational Mathematics}, 9(3):317--334, 2009.

\bibitem[ORS15]{oymak15}
S.~Oymak, B.~Recht, and M.~Soltanolkotabi.
\newblock Sharp time-data tradeoffs for linear inverse problems.
\newblock {\em arXiv}, arXiv:1507.04793, 2015.

\bibitem[Rec11]{recht2011simpler}
B.~Recht.
\newblock A simpler approach to matrix completion.
\newblock {\em Journal of Machine Learning Research}, 12:3413--3430, 2011.

\bibitem[RFP10]{recht10}
B.~Recht, M.~Fazel, and P.~A. Parrilo.
\newblock Guaranteed minimum-rank solutions of linear matrix equations via
  nuclear norm minimization.
\newblock {\em SIAM Review}, 52(3):471--501, 2010.

\bibitem[RR13]{recht13}
B.~Recht and C.~R{\'{e}}.
\newblock Parallel stochastic gradient algorithms for large-scale matrix
  completion.
\newblock {\em Mathematical Programming Computation}, pages 201--226, 2013.

\bibitem[RS05]{Rennie05}
J.~Rennie and N.~Srebro.
\newblock Fast maximum margin matrix factorization for collaborative
  prediction.
\newblock In {\em ICML}, 2005.

\bibitem[Ruh74]{Ruhe74}
A.~Ruhe.
\newblock Numerical computation of principal components when several
  observations are missing.
\newblock Technical report, University of Umea, Institute of Mathematics and
  Statistics Report, 1974.

\bibitem[SL15]{sun14}
R.~Sun and Z.~Luo.
\newblock Guaranteed matrix completion via non-convex factorization.
\newblock In {\em FOCS}, 2015.

\bibitem[Sol14]{soltanolkotabi2014algorithms}
M.~Soltanolkotabi.
\newblock {\em Algorithms and Theory for Clustering and Nonconvex Quadratic
  Programming}.
\newblock PhD thesis, Stanford University, 2014.

\bibitem[SOR15]{desa14}
C.~De Sa, K.~Olukotun, and C.~R{\'{e}}.
\newblock Global convergence of stochastic gradient descent for some nonconvex
  matrix problems.
\newblock In {\em ICML}, 2015.

\bibitem[TG07]{OMP2}
J.~A. Tropp and A.~C. Gilbert.
\newblock Signal recovery from random measurements via orthogonal matching
  pursuit.
\newblock {\em IEEE Transactions on Information Theory}, 53(12):4655--4666,
  2007.

\bibitem[ZL15]{zheng15}
Q.~Zheng and J.~Lafferty.
\newblock A convergent gradient descent algorithm for rank minimization and
  semidefinite programming from random linear measurements.
\newblock In {\em NIPS}, 2015.

\bibitem[ZWL15]{zhao15}
T.~Zhao, Z.~Wang, and H.~Liu.
\newblock Nonconvex low rank matrix factorization via inexact first order
  oracle, 2015.
\newblock \url{http://www.princeton.edu/~zhaoran/papers/LRMF.pdf}.

\end{thebibliography}
\bibliographystyle{alpha}
}
\appendix

\fi

\end{document}